\documentclass{amsart}

\usepackage{amsmath}
\usepackage{amssymb}
\usepackage{bibentry}
\usepackage{enumitem}
\usepackage{graphicx}
\usepackage{mathrsfs}
\usepackage{tabu}
\usepackage[all]{xy}\usepackage{tikz}
\usepackage[comma,numbers,sort,square]{natbib}
\usepackage[colorlinks=true,
            linktocpage=true,
            linkcolor=magenta,
            citecolor=magenta,
            urlcolor=magenta]{hyperref}
\PassOptionsToPackage{hyphens}{url}
\usepackage{cleveref}

\usepackage{doi}

\usetikzlibrary{arrows,shapes,snakes,decorations.pathreplacing,calligraphy}

\newtheorem{theorem}{Theorem}[section]
\newtheorem{proposition}[theorem]{Proposition}
\newtheorem{lemma}[theorem]{Lemma}
\newtheorem{corollary}[theorem]{Corollary}
\theoremstyle{definition}
\newtheorem{definition}[theorem]{Definition}

\newtheorem{statement}[theorem]{Statement}
\newtheorem{question}[theorem]{Question}

\DeclareMathOperator{\res}{\upharpoonright}
\newcommand{\ran}{\operatorname{ran}}
\newcommand{\dom}{\operatorname{dom}}
\renewcommand{\mod}{\text{ }\textrm{mod}\text{ }}

\newcommand{\seq}[1]{\langle #1 \rangle}

\newcommand{\set}[1]{\{ #1 \}}

\renewcommand{\setminus}{\smallsetminus}
\renewcommand{\emptyset}{\varnothing}

\newcommand{\emptystring}{\text{$\langle\rangle$}}

\newcommand{\RCA}{\mathsf{RCA}}

\newcommand{\RT}{\mathsf{RT}}

\newcommand{\TT}{\mathsf{TT}}

\newcommand{\D}{\mathsf{D}}

\newcommand{\uequiv}{\equiv_{\text{\upshape W}}}
\newcommand{\ured}{\leq_{\text{\upshape W}}}
\newcommand{\Ured}{<_{\text{\upshape W}}}
\newcommand{\nured}{\nleq_{\text{\upshape W}}}

\newcommand{\rank}{\mathrm{rk}}
\newcommand{\height}{\mathrm{ht}}
\newcommand{\block}{\mathrm{bl}}
\newcommand{\fopart}{{}^1}
\newcommand{\converges}{\downarrow}
\newcommand{\V}[1]{\mathsf{V}_{#1}}

\begin{document}

\title{The tree pigeonhole principle in the Weihrauch degrees}

\author{Damir D. Dzhafarov}
\address{Department of Mathematics\\
University of Connecticut\\
Storrs, Connecticut U.S.A.}
%341 Mansfield Road\\ Storrs, Connecticut 06269-1009 U.S.A.}
%\curraddr{}
\email{damir@math.uconn.edu}

\author{Reed Solomon}
\address{Department of Mathematics\\
University of Connecticut\\
Storrs, Connecticut U.S.A.}
%341 Mansfield Road\\ Storrs, Connecticut 06269-1009 U.S.A.}
%\curraddr{}
\email{solomon@math.uconn.edu}

\author{Manlio Valenti}
\address{Department of Mathematics\\
University of Wisconsin\\
Madison, Wisconsin U.S.A.}
\curraddr{Department of Computer Science\\
Swansea University\\
Swansea, U.K.}
\email{manliovalenti@gmail.com}

\thanks{Dzhafarov and Solomon were partially supported by a Focused Research Group grant from the National Science Foundation of the United States, DMS-1854355. The authors thank Leszek Ko\l{}odziejczyk, Ludovic Levy Patey, and Arno Pauly for valuable discussions at various stages of the writing of this article. They also thank the two anonymous referees for helpful comments and suggestions.}

\begin{abstract}
We study versions of the tree pigeonhole principle, $\TT^1$, in the context of Weihrauch-style computable analysis. The principle has previously been the subject of extensive research in reverse mathematics, an outstanding question of which investigation is whether $\TT^1$ is $\Pi^1_1$-conservative over the ordinary pigeonhole principle, $\RT^1$. Using the recently introduced notion of the first-order part of an instance-solution problem, we formulate the analogue of this question for Weihrauch reducibility, and give an affirmative answer. In combination with other results, we use this to show that unlike $\RT^1$, the problem $\TT^1$ is not Weihrauch requivalent to any first-order problem. Our proofs develop new combinatorial machinery for constructing and understanding solutions to instances of $\TT^1$.
\end{abstract}

\maketitle

\section{Introduction}\label{sec:intro}

The classical pigeonehole principle, that an infinite set partitioned into finitely many parts must contain at least one part that is infinite, is elementary yet deeply fascinating in terms of its logical content.  In computability theory and proof theory in particular, where the focus is on partitions of the set of natural numbers, it has been a subject of considerable interest. Formally, the statement is as follows, where we use $\omega$ to denote $\set{0,1,2,\ldots}$.

\begin{statement}[Infinitary pigeonhole principle]\label{def:IPHP}
	For every $k \geq 1$ and every $f : \omega \to k$, there exists $i < k$ such that $f^{-1}(\set{i})$ is infinite.
\end{statement}

\noindent In the study of models of arithmetic and reverse mathematics, the principle is non-trivial because of the possibility that the set of natural numbers, and the number $k$ above, may be non-standard. In computable analysis, on the other hand, it is non-trivial because of the fact that the number $i < k$ cannot, in general, be found by any uniform algorithm. These features have inspired extensive investigations into the computational and proof-theoretic aspects of the infinitary pigeonhole principle, and in its relationship to various other logical and mathematical statements. Standard references on reverse mathematics include Simpson~\cite{Simpson-2009}, Hirschfeldt~\cite{Hirschfeldt-2014}, and Dzhafarov and Mummert~\cite{DM-2022}. In \Cref{sec:background}, we review some rudiments of computable analysis in the style of Weihrauch (\cite{Weihrauch-1987,Weihrauch-1992}), which will be our focus here. We refer to~\cite[Chapter 4]{DM-2022} for more details and examples, and to Brattka, Gherardi, and Pauly~\cite{BGP-2021} for a comprehensive overview of the subject. For a general introduction to computability theory, see Soare~\cite{Soare-2016} and Downey and Hirschfeldt~\cite{DH-2009}.

A further point of interest is that the infinitary pigeonhole principle may be formulated as a special case of Ramsey's theorem.

\begin{statement}[Ramsey's theorem for singletons]
	For every $k \geq 1$ and every $f : \omega \to k$, there exists an infinite set $M \subseteq \omega$ monochromatic for $f$ (i.e., such that $f \res M$ is constant).
\end{statement}

\noindent We denote this statement by $\RT^1$.  The equivalence of the two statements is obvious: given $i < k$ such that $f^{-1}(\set{i})$ is infinite we have that $M = \set{x \in \omega: f(x) = i}$ is an infinite monochromatic set for $f$, and given $M$ we have that $f^{-1}(\set{\min M})$ is infinite. Importantly, this equivalence is uniform in the sense of Weihrauch reducibility (see, e.g.,~\cite{DM-2022}, Proposition 4.3.4).

In this paper, we are interested in a version of the pigeonhole principle not for partitions of $\omega$ but for partitions of trees. To state it, we begin with some definitions. We let $\omega^{<\omega}$ denote the set of all finite strings of natural numbers (i.e., functions $n \to \omega$ for some $n \in \omega$), and $2^{<\omega}$ the subset of $\omega^{<\omega}$ of all finite strings of $0$s and $1$s (i.e., functions $n \to 2$ for some $n \in \omega$). We will typically use lowercase Greek letters near the beginning of the alphabet ($\alpha,\beta,\gamma,\ldots$) for general elements of $\omega^{<\omega}$, and lowercase Greek letters near the middle of the alphabet ($\sigma,\tau,\rho,\ldots$) for elements specifically of $2^{<\omega}$. For $n \in \omega$, we let $\omega^{< n}$ and $2^{< n}$ denote the respective restrictions of $\omega^{<\omega}$ and $2^{<\omega}$ to strings of length less than $n$. We denote the length of a string $\alpha \in \omega^{<\omega}$ by $|\alpha|$, and let $\emptystring$ denote the empty string (i.e., the unique string of length $0$). We use $\preceq$ for the reflexive prefix relation on $\omega^{<\omega}$, and $\prec$ for the irreflexive. For $\alpha,\beta \in \omega^{<\omega}$ we write $\alpha \mid \beta$ if $\alpha \npreceq \beta$ and $\beta \npreceq \alpha$. If $\alpha \mid \beta$ then $\alpha$ and $\beta$ are called \emph{incomparable}, and otherwise they are called \emph{comparable}. We write $\alpha\beta$ for the concatenation of $\alpha$ by $\beta$ (i.e., the string $\gamma$ of length $|\alpha|+|\beta|$ with $\gamma(x) = \alpha(x)$ for all $x < |\alpha|$ and $\gamma(x) = \beta(x-|\alpha|)$ for $|\alpha| \leq x < |\alpha|+|\beta|$). For $s \in \omega$, we write $0^s$ for the string $\alpha$ of length $s$ with $\alpha(x) = 0$ for all $x < s$, and similarly for $1^s$.

\begin{definition}
	Fix $S \subseteq 2^{<\omega}$.
	\begin{enumerate}
		%\item We write $S \cong 2^{<\omega}$ if $(S,\preceq)$ and $(2^{<\omega},\preceq)$ are isomorphic as partial orders.
		\item For $n \in \omega \cup \set{\omega}$, we write $S \cong 2^{< n}$ if $(S,\preceq)$ and $(2^{< n},\preceq)$ are isomorphic as partial orders.
		\item For $k \geq 1$, a \emph{$k$-coloring of $S$} is a map $f : S \to k = \set{0,\ldots,k-1}$.
		\item A \emph{coloring of $S$} is a $k$-coloring of $S$ for some $k \geq 1$.
		\item A set $H \subseteq S$ is \emph{monochromatic} for a coloring $f$ of $S$ if $f$ is constant on $H$. We call $f(\min H)$ the \emph{color of $H$ under $f$}.
	\end{enumerate} 
\end{definition}

\noindent The focus of our investigation is the following principle, originally formulated by Chubb, Hirst, and McNicholl~\cite{CHM-2009}.

\begin{statement}[Tree pigeonhole principle]\label{def:TPHP}
	For every $k \geq 1$ and every $f : 2^{<\omega} \to k$, there exists $H \cong 2^{<\omega}$ that is monochromatic for $f$.
\end{statement}

\noindent We denote this statement by $\TT^1$. Its proof is straightforward, though less trivial than that of $\RT^1$. Consider a coloring $f : 2^{<\omega} \to 2$, for simplicity. If $\set{\tau \in 2^{<\omega} : f(\tau) = 0}$ is dense (meaning that for every $\sigma \in 2^{<\omega}$ there exists $\tau \succeq \sigma$ with $f(\tau) = 0$) then an $H \cong 2^{<\omega}$ monochromatic for $f$ with color $0$ can be easily constructed by unbounded search. Otherwise, there exists $\sigma \in 2^{<\omega}$ such that $f(\tau) = 1$ for all $\tau \succeq \sigma$, and then $H = \set{\tau \in 2^{<\omega} : \tau \succeq \sigma} \cong 2^{<\omega}$ is monochromatic for $f$ with color $1$. Such an $H$ is sometimes called a monochromatic \emph{cone}, leading to this proof being called a ``dense-or-cone'' argument. More generally, given $f : 2^{<\omega} \to k$, an iteration of this argument reveals an $i < k$ and a $\sigma \in 2^{<\omega}$ such that $\set{\tau \in 2^{<\omega} : f(\tau) = i}$ is dense above $\sigma$, from which an $H \cong 2^{<\omega}$ monochromatic for $f$ with color $i$ may be constructed. We will explore variations on ``dense-or-cone'' arguments in \Cref{sec:dense-or-cone}.

In the literature, $\TT^1$ and its generalizations to higher dimensions have collectively come to be called the \emph{tree theorem}. However, it is worth emphasizing that the sets $H \cong 2^{<\omega}$ it deals with are not trees in the usual sense used in computability theory. More precisely, a set $H \cong 2^{<\omega}$ need not be closed downward under $\preceq$. The notion also differs from the usual sense of trees used in descriptive set theory, in that such an $H$ need not preserve meets (i.e., least common prefixes). A partition principle better suited to the latter type of structure is the \emph{Halpern-L\"{a}uchli theorem} (or \emph{Milliken's tree theorem}, as it is known in higher dimensions) which generalizes the Chubb-Hirst-McNicholl version we will focus on here. (See Todor\v{c}evi\'{c}~\cite{Todorcevic-2010} and Dobrinen~\cite{Dobrinen-2018} for more details, and Angl\'{e}s d'Auriac, et al.~\cite{ACDMP-2022} for a recent computability-theoretic analysis.) We will use the word ``tree'' exclusively in the computability-theoretic sense throughout this paper, and refer to \Cref{def:TPHP} and its variants (introduced below) only by $\TT^1$ and similar initialisms.

It is well-known that $\RT^1$ is a consequence of $\TT^1$. The following argument, which we include for completeness, was first noted in~\cite[proof of Theorem 1.5]{CHM-2009}. Given $g : \omega \to k$, define $f : 2^{<\omega} \to k$ by $f(\sigma) = g(|\sigma|)$. By the tree pigeonhole principle, we may fix an $H \cong 2^{<\omega}$ monochromatic for $f$, say with color $i < k$. Then $\set{|\sigma| : \sigma \in H}$ is an infinite monochromatic set for $g$, so in particular, $g^{-1}(\set{i})$ is infinite.

Finding the precise relationship between $\RT^1$ and $\TT^1$ is an ongoing investigation using the framework of reverse mathematics. Here, $\RT^1$ and $\TT^1$ are formalized in the language of second-order arithmetic, and their strengths are calibrated in terms of implications over the weak base theory $\RCA_0$ (or sometimes, even weaker systems). By a well-known result of Hirst~\cite[Theorem 6.4]{Hirst-1987}, $\RT^1$ is equivalent over $\RCA_0$ to the bounding scheme $\mathsf{B}\Sigma^0_2$ (which, by classic results, itself lies strictly in-between the induction schemes $\mathsf{I}\Sigma^0_1$ and $\mathsf{I}\Sigma^0_2$). Corduan, Groszek, and Mileti~\cite[Section 3.3]{CGM-2010} obtained the surprising result that $\TT^1$ lies strictly above $\mathsf{B}\Sigma^0_2$, while Chong, Li, Wei, and Yang~\cite[Corollary 4.2]{CLWY-2020} proved it is strictly below $\mathsf{I}\Sigma^0_2$. Thus, we know that $\TT^1$ is a stronger principle than $\RT^1$, but more questions remain. Chief among these is whether $\TT^1$ is conservative for $\Pi^1_1$ sentences over $\RCA_0 + \RT^1$. In particular, it is possible that the first-order part of $\TT^1$ (i.e., the set of its first-order consequences over $\RCA_0$) is exactly (the set of such consequences of) $\RT^1$.

In this article, we investigate the strength of $\TT^1$, and its relationship to $\RT^1$, using the tools of Weihrauch reducibility and Weihrauch-style computable analysis. This is a different framework from reverse mathematics, but one that is often complementary (see Dorais, et al.,~\cite[Section 1]{DDHMS-2016} for a discussion). In many cases, the results in this approach provide a finer, more detailed view of those in classical reverse mathematics. We obtain a number of results along these lines below. Recently, Dzhafarov, Solomon, and Yokoyama~\cite{DSY-TA} introduced an analogue for Weihrauch reducibility of the first-order part of a problem (defined precisely in the next section). We show that, for the most natural way of formulating $\TT^1$ and $\RT^1$ in the context of Weihrauch reducibility, the first-order part of $\TT^1$ in the sense of~\cite{DSY-TA} is precisely $\RT^1$. Thus, we obtain a positive answer to the analogue of the question above. As shown in~\cite{DSY-TA}, and in follow-up works by Sold\`{a} and Valenti~\cite{SV-2022} and Goh, Pauly, and Valenti~\cite{GPV-2021}, this new notion of ``first-order part'' largely agrees with the one from reverse mathematics. Thus, our results may be viewed as evidence in support of the corresponding answers for the original question, too.

Before moving on, we include one final remark for general interest. By its equivalence with the infinitary pigeonhole principle, $\RT^1$ may be regarded as either a second-order statement (asserting the existence of certain kinds of sets) or as a first-order (arithmetical) statement. Interestingly, the same is not true of $\TT^1$, as we learned from Leszek Ko\l{}odziejczyk (private communication). We include the argument he supplied as it is quite lovely but appears to be missing from the literature. As discussed above, Chong, Li, Wei, and Yang~\cite[Corollary 4.2]{CLWY-2020} exhibited a model $(M,\mathcal{S})$ of $\RCA_0 + \neg \mathsf{I}\Sigma^0_2 + \TT^1$. By results of Corduan, Groszek, and Mileti~\cite[Proposition 3.5]{CGM-2010}, there is an $X \in \mathcal{S}$ such that some $\TT^1$ instance $f$ computable from $X$ (in the sense of $M$) has no solution computable from $X$. Consider the topped model $M[X]$, which has first-order part $M$ and second order part the set of all $Y \subseteq M$ computable from $X$. Then $\RCA_0$ holds in $M[X]$ but $\TT^1$ fails, as witnessed by $f$. But $(M,\mathcal{S})$ and $M[X]$ have the same first-order part (namely, $M$) and thus satisfy the same $\Pi^1_1$ sentences. Any $\Pi^1_1$ statement equivalent over $\RCA_0$ to $\TT^1$ would thus hold in $(M,\mathcal{S})$ and hence in $M[X]$, which cannot be since $\TT^1$ is true in one but not the other. (Below, we obtain the analogue of this result in the setting of Weihrauch reducibility, but using very different means.)

\section{Preliminaries}\label{sec:background}

Our computability-theoretic notation is largely standard, following, e.g., Downey and Hirschfeldt~\cite{DH-2010} and Soare~\cite{Soare-2016}. All objects under discussion are assumed to be represented by elements of $2^{\omega}$. To this end, subsets of $\omega$ are identified with their characteristic functions, and elements of $\omega$ are identified with singleton subsets of $\omega$. We fix a standard computable bijection between $\omega^{<\omega}$ and $\omega$ such that for all $\alpha \in \omega^{<\omega}$, $|\alpha|$ is uniformly computable from $\alpha$. A subset of $\omega^{<\omega}$ is then identified with its image under this bijection. Finite sets will usually be specified by (and identified with) their canonical indices. More complex representations will be defined as needed. 

For subsets $A_0,\ldots,A_{n-1}$ of $\omega$ (or of $\omega^{<\omega}$, via our representation) we write $\seq{A_0,\ldots,A_{n-1}}$ instead of the more common $A_0 \oplus \cdots \oplus A_{n-1}$. And given a Turing functional $\Phi$, we write $\Phi(A_0,\ldots,A_{n-1})$ as shorthand for $\Phi^{\seq{A_0,\ldots,A_{n-1}}}$.

The principal object of interest in the framework of Weihrauch-style computable analysis is the following.

\begin{definition}
	A \emph{instance-solution problem} (or just \emph{problem}) is a partial multifunction $\mathsf{P} : \subseteq X \rightrightarrows Y$, where $X$ and $Y$ are subsets of $\omega^\omega$, such that $\mathsf{P}(x) \neq \emptyset$ for each $x \in \dom(\mathsf{P})$. The elements of $\dom(\mathsf{P})$ are called the \emph{instances of $\mathsf{P}$} (or \emph{$\mathsf{P}$-instances}) and for each $x \in \dom(\mathsf{P})$, the elements of $\mathsf{P}(x)$ are called the \emph{solutions to $x$ in $\mathsf{P}$} (or \emph{$\mathsf{P}$-solutions to $x$}).
\end{definition}

\noindent For many principles, there is a natural way to move between their formalization as a statement of second-order arithmetic and as a problem. For example, we can consider the following problem forms of restrictions of $\RT^1$ and $\TT^1$.

\begin{definition}\label{def:kversions}
	Fix $k \geq 1$.
	\begin{enumerate}
		\item $\RT^1_k$ is the problem whose instances are all colorings $f : \omega \to k$, with the solutions to any such $f$ being all infinite sets $M \subseteq \omega$ monochromatic for $f$.
		\item $\TT^1_k$ is the problem whose instances are all colorings $f : 2^{<\omega} \to k$, with the solutions to any such $f$ being all $H \cong 2^{<\omega}$ monochromatic for $f$.
	\end{enumerate}
\end{definition}

\noindent For $\RT^1$ and $\TT^1$, there is some subtlety. Brattka and Rakotoniaina~\cite[Definition 3.1]{BR-2017} considered two problem forms of $\RT^1$, denoted by $\RT^1_+$ and $\RT^1_{\mathbb{N}}$, which we define next. We define two problem forms of $\TT^1$ in an analogous fashion.

\begin{definition}\label{def:+Nversions}
	\
	\begin{enumerate}
		\item $\RT^1_+$ is the problem whose instances are all pairs $\seq{k,f}$ where $k \geq 1$ and $f$ is a coloring $\omega \to k$, with the solutions to any such pair being all infinite sets $M \subseteq \omega$ monochromatic for $f$.
		\item $\RT^1_{\mathbb{N}}$ is the problem whose instances are all colorings $f : \omega \to k$ for some $k \geq 1$, with the solutions to any such $f$ being all infinite sets $M \subseteq \omega$ monochromatic for $f$.
		\item $\TT^1_+$ is the problem whose instances are all pairs $\seq{k,f}$ where $k \geq 1$ and $f$ is a coloring $2^{<\omega} \to k$, with the solutions to any such pair being all $H \cong 2^{<\omega}$ monochromatic for $f$.
		\item $\TT^1_{\mathbb{N}}$ is the problem whose instances are all colorings $f : 2^{<\omega} \to k$ for some $k \geq 1$, with the solutions to any such $f$ being all $H \cong 2^{<\omega}$ monochromatic for $f$.
	\end{enumerate}
\end{definition}

\noindent The distinction here is whether an instance is just a coloring with bounded range, or whether an upper bound must be explicitly provided along with the coloring. From the point of view of a classical system like $\RCA_0$, there is no difference between these two formulations. But an upper bound cannot be obtained uniformly computably from a coloring, and this is often useful in constructions---i.e., the ability to increase the number of colors finitely many times while building a coloring without declaring an upper bound in advance. (See, e.g., the proof of \Cref{thm:TT1k_not_to_RT1N} below, and see~\cite[Section 3.3]{DM-2022} for a discussion of this and similar issues.) From a purely computability-theoretic perspective, then, the versions $\RT^1_{\mathbb{N}}$ and $\TT^1_{\mathbb{N}}$ are more natural. (For more on the distinction between the $+$ and $\mathbb{N}$ cases, see the discussion in~\cite{PPS-TA}.)

Problems can be compared using the following reducibility notion, originally proposed by Weihrauch~\cite{Weihrauch-1987,Weihrauch-1992}, and independently by Dorais, et al.~\cite{DDHMS-2016} in the guise we use here.

\begin{definition}
	Let $\mathsf{P}$ and $\mathsf{Q}$ be problems.
	\begin{enumerate}
		\item $\mathsf{P}$ is \emph{Weihrauch reducible} to $\mathsf{Q}$, written $\mathsf{P} \ured \mathsf{Q}$, if there exist Turing functionals $\Phi$ and $\Psi$ such that $\Phi(x) \in \dom(\mathsf{Q})$ for every $x \in \dom(\mathsf{P})$, and $\Psi(x,y) \in \mathsf{P}(x)$ for every $y \in \mathsf{Q}(\Phi(x))$. In this case, we also say that $\mathsf{P}$ is Weihrauch reducible to $\mathsf{Q}$ \emph{via $\Phi$ and $\Psi$}.
		\item If $\mathsf{P} \ured \mathsf{Q}$ and $\mathsf{Q} \ured \mathsf{P}$ then $\mathsf{P}$ and $\mathsf{Q}$ are \emph{Weihrauch equivalent}, written $\mathsf{P} \uequiv \mathsf{Q}$.
	\end{enumerate}
\end{definition}

\noindent The argument given in \Cref{sec:intro} for $\RT^1$ being a consequence of $\TT^1$ directly shows that for all $k \geq 1$, $\RT^1_k \ured \TT^1_k$, that $\RT^1_+ \ured \RT^1_{\mathbb{N}} \ured \TT^1_{\mathbb{N}}$, and also that $\RT^1_+ \ured \TT^1_+$. That none of these reductions can be reversed follows from \Cref{1TTnuredRT1,thm:TT1k_not_to_RT1N} below. We will continue to use the initialisms $\RT^1$ and $\TT^1$ when speaking, generally, about any of the problems from \Cref{def:kversions,def:+Nversions}.

We now define the notion of the \emph{first-order part} of a problem, originally introduced by Dzhafarov, Solomon, and Yokoyama~\cite[Definition 2.1]{DSY-TA}. In what follows, Turing functionals are assumed to be represented by their indices within some fixed enumeration of all partial computable oracle functions.

\begin{definition}
	Let $\mathsf{P}$ be a problem. The first-order part of $\mathsf{P}$, denoted $\fopart \mathsf{P}$, is the following problem:
	\begin{itemize}
		\item the instances of $\fopart \mathsf{P}$ are all triples $\seq{A,\Delta,\Gamma}$ where $A \in 2^\omega$ is arbitrary and $\Delta$ and $\Gamma$ are Turing functionals such that $\Delta(A) \in \dom(\mathsf{P})$ and $\Gamma(A,y)(0) \converges$ for all $y \in \mathsf{P}(\Delta(A))$;
		\item the solutions to any $\seq{A,\Delta,\Gamma}$ as above are all values of the form $\Gamma(A,y)(0)$ for some $y \in \mathsf{P}(\Delta(A))$.
	\end{itemize}	
\end{definition}

\noindent Intuitively, $\fopart \mathsf{P}$ considers an instance of $\mathsf{P}$ along with a functional that halts on an initial segment of any $\mathsf{P}$-solution, and accepts the output of that computation as its own solution. (A slightly different but equivalent formulation appears in~\cite[Definition 3.1]{SV-2022}.)

The reason for the name ``first-order'' is that, as shown in~\cite[Theorem 2.2]{DSY-TA}, $\fopart \mathsf{P}$ is Weihrauch equivalent to the maximum under $\ured$ of all $\mathsf{Q} \ured \mathsf{P}$ with co-domain $\omega$. Such a problem is called \emph{first-order}. Thus, $\fopart \mathsf{P}$ is the ``strongest'' first-order problem that lies Weihrauch below $\mathsf{P}$. Using this characterization, we can thus see that for all $k \geq 1$, $\fopart \RT^1_k \uequiv \RT^1_k$, and that $\fopart \RT^1_{+} \uequiv \RT^1_{+}$ and $\fopart \RT^1_{\mathbb{N}} \uequiv \RT^1_{\mathbb{N}}$. This is because in each case, knowing the color of an infinite homogeneous for a given coloring suffices to uniformly compute such a set from the coloring. This is not the case for trees, as we will show below. But we do have that for all $k \geq 1$, $\RT^1_k \ured \fopart \TT^1_k \ured \TT^1_k$; that $\RT^1_{+} \ured \fopart \TT^1_{+} \ured \TT^1_{+}$; and that $\RT^1_{\mathbb{N}} \ured \fopart \TT^1_{\mathbb{N}} \ured \TT^1_{\mathbb{N}}$.

We conclude this section with some definitions specific to our working with elements and subset of $2^{<\omega}$.

\begin{definition}
	Fix $S \subseteq 2^{<\omega}$.
	\begin{enumerate}
		\item The \emph{rank of $\sigma \in S$ in $S$}, denoted $\rank_S(\sigma)$, is $|\set{\tau \in S : \tau \prec \sigma}|$.
		\item For $n \in \omega$, $S^{\rank < n}$ denotes the set $\set{\sigma \in S : \rank_S(\sigma) < n}$.
		\item $S$ is \emph{rooted} if it there is a unique $\rho \in S$, called the \emph{root} of $S$, with $\rank_S(\rho) = 0$.
		\item $\tau \in S$ is a \emph{successor of $\sigma \in S$ in $S$} if $\sigma \prec \tau$ and $\rank_S(\tau) = \rank_S(\sigma)+1$.
		\item $\sigma \in S$ is a \emph{leaf of $S$} if it has no successor in $S$.
		\item $S$ is \emph{symmetric} if all $\sigma,\tau \in S$ with $\rank_S(\sigma) = \rank_S(\tau)$ have the same number of successors in $S$.
		\item The \emph{height} of $S$, denoted $\height(S)$, is $\sup \set{ \rank_S(\sigma) + 1 : \sigma \in S}$. (So $\height(\emptyset) = 0$.)
	\end{enumerate}
\end{definition}

For all oracles $A$, we assume that if $\Phi(A)(n) \converges$ then $\Phi(A)(m) \converges$ for all $m < n$. We also adopt the following more specific conventions about oracles contained in $2^{<\omega}$. If $\Phi$ is a functional and $S \subseteq 2^{<\omega}$ is a finite set such that $\Phi(S)(n) \converges$ for some $n$, then we assume the length of the computation is bounded by the length of the shortest leaf in $S$. This ensures that if $T \subseteq 2^{<\omega}$ is any set (finite or infinite) such that $T \supseteq S$ and every $\sigma \in T \setminus S$ extends a leaf of $S$, then $\Phi(T)(n) \converges = \Phi(S)(n)$. More generally, we have that if $\varphi(X)$ is any $\Sigma^0_1$ formula such that $\varphi(S)$ holds, then so does $\varphi(T)$.

Unless otherwise noted, all objects throughout should be regarded as elements of $\omega^{\omega}$ via standard codings.

\section{The complexity of building monochromatic solutions}\label{sec:dense-or-cone}

An initial point of distinction between the complexities of $\RT^1$ and $\TT^1$ can be seen in the complexity of checking whether a given set \emph{is} a solution to a given coloring, versus whether a given element is \emph{extendible} to a solution. It is easy to see that for $\RT^1$, both questions are $\Pi^0_2$ relative to the set and the coloring. For $\TT^1$, checking whether a given set is a solution is likewise $\Pi^0_2$: indeed, saying that $H$ is monochromatic for a given $\TT^1$ instance $f$ is $\Pi^0_1$ in $H$ and $f$, while saying that $H \cong 2^{<\omega}$ can be expressed as
\begin{eqnarray*}
	H \neq \emptyset \wedge (\forall \sigma \in H)(\exists \tau_0,\tau_1 \in H)[\sigma \prec \tau_0,\tau_1 \wedge \tau_0 \mid \tau_1]\hspace{5em}\\
	\wedge (\forall \sigma \in H)(\forall \tau_0,\tau_1,\tau_2 \in H)[\sigma \prec \tau_0,\tau_1,\tau_2 \to (\exists \rho \in H)(\exists i < 3)[\sigma \prec \rho \prec \tau_i]].
\end{eqnarray*}

\noindent However, in stark contrast, whether or not an element of $2^{<\omega}$ is part of a solution to a given instance of $\TT^1$ is far harder: it is $\Sigma^1_1$-complete relative to $H$ and $f$. We can express this using Weihrauch reducibility using the definition below. As a canonical representative of a ``$\Sigma^1_1$-complete problem'' we take $\mathsf{WF}$, which we recall is the problem whose instances are arbitrary trees $T \subseteq \omega^{<\omega}$, with the solution to any such $T$ being $0$ or $1$ depending as $T$ is ill-founded or well-founded, respectively.

\begin{definition}
	Fix $k \geq 1$. $\TT^1_k$-$\mathsf{Ext}$ is the problem whose instances are pairs $\seq{f,\sigma}$ where $f$ is a coloring $2^{<\omega} \to k$ and $\sigma \in 2^{<\omega}$, with the solutions to any such pair being $1$ or $0$ depending as there is or is not an $H \cong 2^{<\omega}$ monochromatic for $f$ and containing $\sigma$.
\end{definition}

\begin{theorem}
	For all $k \geq 2$, $\TT^1_k$-$\mathsf{Ext} \uequiv \mathsf{WF}$.
\end{theorem}

\begin{proof}
	The reduction from $\TT^1_k$-$\mathsf{Ext}$ to $\mathsf{WF}$ is straightforward. Given an instance $\seq{f,\sigma}$ of the former, checking whether there exists $H \cong 2^{<\omega}$ monochromatic for $f$ and containing $\sigma$ as an initial segment is a uniformly $\Sigma^1_1$ property in $f$ and $\sigma$, and thus can be decided using $\mathsf{WF}$.
	
	In the other direction, let $T \subseteq \omega^{<\omega}$ be a tree. We define a uniformly $T$-computable coloring $f : 2^{<\omega} \to 2$ with $f(\emptystring) = 0$ and such that $T$ is well-founded if and only if there is no $H \cong 2^{<\omega}$ monochromatic for $f$ with color $0$. Thus, $T$ is well-founded if and only if the $\TT^1_2$-$\mathsf{Ext}$-solution to $\seq{f,\set{\emptystring}}$ is $0$. Define $S = T * \omega^{<\omega} = \set{ \seq{\alpha,\beta} : \alpha \in T, \beta \in \omega^{<\omega},|\alpha|=|\beta|}$, viewed as a subtree of $\omega^{<\omega}$ by thinking of $\seq{\alpha,\beta}$ as a prefix of $\seq{\alpha',\beta'}$ if and only if $\alpha \preceq \alpha'$ and $\beta \preceq \beta'$. Observe that $[T] \neq \emptyset$ if and only if $|[S]| = 2^{\aleph_0}$. Let $R$ be the collection of all strings of the form $0^{\gamma(0)}10^{\gamma(1)}1\cdots 0^{\gamma(|\gamma|-1)}1 0^s$ for some $\gamma \in S$ and some $s \in \omega$. Then $R$ is a $T$-computable subtree of $2^{<\omega}$, and we have $\emptystring \in R$. Although $[R]$ is non-empty (since it contains, at the very least, the constantly $0$ function) we have that $|[R]| = 2^{\aleph_0}$ if and only if $|[S]| = 2^{\aleph_0}$. We define $f$ by setting $f(\sigma) = 0$ if $\sigma \in R$ and $f(\sigma) = 1$ if $\sigma \notin R$. Thus $f(\emptystring) = 0$, and there exists $H \cong 2^{<\omega}$ monochromatic for $f$ and containing $\emptystring$ if and only if there exists a $H \cong 2^{<\omega}$ monochromatic for $f$ with color $0$, which in turn holds in if and only if	$|[R]| = 2^{\aleph_0}$, which, by our previous observations, holds if and only if $[T] \neq \emptyset$. This completes the proof.
\end{proof}

The theorem highlights the intrinsic complexity of knowing whether a given instance of $\TT^1$ admits a solution of a particular, fixed color. In the ``dense-or-cone'' proof of $\TT^1$ discussed in \Cref{sec:intro}, deciding the color in which to try to build a solution was handled by asking certain density questions about the coloring. In an attempt to better understand this, we now look at several variants of the ``dense-or-cone'' argument, and investigate how they relate. Although our discussion could be made more general, and apply to $\TT^1$ for arbitrary numbers of colors, we restrict our discussion to $2$-colorings. Our goal is not to obtain a complete classification of these problems, which are not overly interesting in their own, but rather to get a sense of the relationships between different approaches to solving $\TT^1$, as a heuristic for our work moving forward.

\begin{definition} We define the following problems. Each has as its set of instances all $2$-colorings of $2^{<\omega}$. Thus, we specify a problem by the set of solutions to a given such coloring $f : 2^{<\omega} \to 2$.
	\begin{itemize}
		\item $\V0$: The solutions to $f$ are all pairs $\seq{i,\sigma} \in 2 \times 2^{<\omega}$ such that $\set{\tau : f(\tau) = i}$ is dense above $\sigma$.
		\item $\V1$: The solutions to $f$ are all pairs $\seq{i,\sigma} \in 2 \times 2^{<\omega}$ such that either $\set{\tau : f(\tau) = 0}$ and $\set{\tau : f(\tau) = 1}$ are both dense, or $f(\tau) = i$ for all $\tau \succeq \sigma$.
		\item $\V2$: The solutions to $f$ are all pairs $\seq{i,\sigma}$ such that if $i = 0$ then $\set{\tau : f(\tau) = 0}$ is dense, and if $i = 1$ then $\set{\tau : f(\tau) = 1}$ is dense above $\sigma$.
		\item $\V3$: The solutions to $f$ are all pairs $\seq{i,\sigma}$ such that if $i = 0$ then $\set{\tau : f(\tau) = 0}$ is dense, and if $i = 1$ then all $\tau \succeq \sigma$ with $f(\tau) = 0$ are comparable.
		\item $\V4$: The solutions to $f$ are all pairs $\seq{i,\sigma}$ such that if $i = 0$ then $\set{\tau : f(\tau) = 0}$ is dense, and if $i = 1$ then $f(\tau) = 1$ for all $\tau \succeq \sigma$.
	\end{itemize}
\end{definition}

It is clear from the definitions that $\TT^1_2 \ured \V0 \ured \V1$ and that $\V2 \ured \V3 \ured \V4$. A bit less straightforward is that also $\V1 \ured \V2$, which follows from \Cref{V1_equiv,V2_equiv} below. The purpose of this section is to examine the remaining reductions between these problems. In fact, we prove the following theorem.

\begin{theorem}\label{thm_Vstuff_MAIN}
	$\TT^1_2 \Ured \V0 \uequiv \V1 \Ured \V2 \uequiv \V3 \uequiv \V4$.
\end{theorem}

\noindent More precisely, we classify $\V0$ and $\V1$, and $\V2, \V3$, and $\V4$, in terms of previously studied problems from the computable analysis literature. We begin by recalling their definitions. Throughout the following, we think of an arbitrary $\mathbf{\Pi}^0_1$ subset $A$ of $\omega$ as being represented by an enumeration of its complement, which we take to be a total function $e : \omega \to \omega$ such that $\overline{A} = \set{x : x+1 \in \ran(e)}$. (Thus, $\ran(e) = \set{0}$ if and only if $A = \omega$.)

\begin{definition}
	We define the following problems.
	\begin{enumerate}
		\item \emph{Closed choice on $\mathbb{N}$} ($\mathsf{C}_{\mathbb{N}}$) is the problem whose instances are all non-empty $\mathbf{\Pi}^0_1$ subsets $A$ of $\omega$, with the solutions to any such $A$ being all $x \in A$.
		\item The \emph{total continuation} of $\mathsf{C}_{\mathbb{N}}$ ($\mathsf{TC}_{\mathbb{N}}$) is the problem whose instances are all $\mathbf{\Pi}^0_1$ subsets $A$ of $\omega$, with the solutions to any such $A$ being all $x \in \omega$ if $A = \emptyset$, and all $x \in A$ otherwise.
		\item The \emph{strong total continuation} of $\mathsf{C}_{\mathbb{N}}$ ($\mathsf{sTC}_{\mathbb{N}}$) is the problem whose instances are all $\mathbf{\Pi}^0_1$ subsets $A$ of $\omega$, with the solutions to any such $A$ being $-1$ if $A = \emptyset$, and all $x \in A$ otherwise.
	\end{enumerate}
\end{definition}

\noindent A full account of closed choice problems is given by Brattka, Gherardi, and Pauly~\cite[Section 7]{BGP-2021}. The total continuation of closed choice was first studied by Neumann and Pauly~\cite{NP-2018}. The strong total continuation, as a general operator on problems, was introduced more recently, by Marcone and Valenti~\cite[Definition 4.17]{MV-2021}. The difference between $\mathsf{C}_{\mathbb{N}}$ and $\mathsf{TC}_{\mathbb{N}}$ is that the latter problem is also defined on the empty set, while the former is not. The difference between $\mathsf{TC}_{\mathbb{N}}$ and $\mathsf{sTC}_{\mathbb{N}}$ is that the latter must indicate, as part of its solution, whether or not the instance it is solving is the empty set, while the former need not do so.

\begin{proposition}\label{V1_equiv}
	$\V0 \uequiv \V1 \uequiv \mathsf{TC}_{\mathbb{N}}$.
\end{proposition}

\begin{proof}
	As already noted, $\V0 \ured \V1$. It thus remains to show that $\V1 \ured \mathsf{TC}_{\mathbb{N}} \ured \V0$. To see that $\V1 \ured \mathsf{TC}_{\mathbb{N}}$, fix any instance $f : 2^{<\omega} \to 2$ of $\V1$. Consider the $\Pi^{0,f}_1$ set $A = \set{\seq{i,\sigma} : (\forall \tau \succeq \sigma)[f(\tau) = i]}$, regarded as an instance of $\mathsf{TC}_{\mathbb{N}}$
	%(under a suitable effective coding)
	uniformly computable from $f$. Clearly, $A \neq \emptyset$ if and only if $\set{\tau : f(\tau) = 0}$ and $\set{\tau : f(\tau) = 1}$ are not both dense. Thus, any $\mathsf{TC}_{\mathbb{N}}$-solution to $A$ is a $\V1$-solution to $f$.
	%If $\set{\tau : f(\tau) = 0}$ and $\set{\tau : f(\tau) = 1}$ are not both dense, then the least pair $\seq{i_0,\sigma_0}$ such that $f(\tau) = i$ for all $\tau \succeq \sigma$ is uniformly computable from $f'$. Hence, we can fix a uniformly $f$-computable function $\ell : \omega \to 2 \times 2^{<\omega}$ such that, if $\set{\tau : f(\tau) = 0}$ and $\set{\tau : f(\tau) = 1}$ are not both dense, then $\lim_s \ell(s)$ exists and equals $\seq{i_0,\sigma_0}$. Define an instance $e : \omega \to \omega$ of $\mathsf{TC}_{\mathbb{N}}$ as follows. Given $s \in \omega$, compute $\ell(s) = \seq{i,\sigma}$, choose the least $\sigma' \succeq \sigma$ such that $\seq{i,\sigma} \notin \ran(e \res s)$, and let $e(s)$ be the least element of $\omega$ different from $\seq{i,\sigma'}$. Now, let $\seq{i,\sigma}$ be any $\mathsf{TC}_{\mathbb{N}}$-solution to $e$. If $\set{\tau : f(\tau) = 0}$ and $\set{\tau : f(\tau) = 1}$ are not both dense then $\ran(e) = \omega \setminus \set{\seq{i_0,\sigma'}}$ for some $\sigma' \succeq \sigma_0$, so $\seq{i,\sigma} = \seq{i_0,\sigma'}$. In particular, $f(\tau) = i = i_0$ for all $\tau \succeq \sigma$ in this case. Thus, $\seq{i,\sigma}$ is a $\V1$-solution to $f$.
	
	To conclude the proof, we show that $\mathsf{TC}_{\mathbb{N}} \ured \V0$. Fix an instance $A$ of $\mathsf{TC}_{\mathbb{N}}$, specified by a co-enumeration $e : \omega \to \omega$. From $e$ we can uniformly compute the function $\ell : \omega \to \omega$ such that $\ell(s) = (\mu x)[x+1 \notin e \res s]$. Clearly, $A \neq \emptyset$ if and only if $\lim_s \ell(s)$ exists, in which case $\lim_s \ell(s) \in A$. We now define a uniformly $e$-computable coloring $f : 2^{<\omega} \to 2$, as follows. Let $f(\emptystring) = 0$, and for each $s \in \omega$, each string $\sigma$ of length $s$, and each $i < 2$ let
	\[
		f(\sigma i\bigskip) =
		\begin{cases}
			i & \text{if } \ell(s) \neq \ell(s+1),\\
			f(\sigma) & \text{otherwise}.
		\end{cases}
	\]
	Now, let $\seq{i,\sigma}$ be any $\V0$-solution to $f$. We claim that $\ell(|\sigma|)$ is a $\mathsf{TC}_{\mathbb{N}}$-solution to $A$. Indeed, suppose $A \neq \emptyset$ and fix the least $s$ such that $\ell(s) = \ell(t)$ for all $t \geq s$. If $|\sigma| < s$, then it follows from the definition of $f$ that neither $\set{\tau : f(\tau) = 0}$ nor $\set{\tau : f(\tau) = 1}$ is dense above $\sigma$, even though we assumed $\set{\tau : f(\tau) = i}$ is. Thus, $|\sigma| \geq s$, which proves the claim.
\end{proof}

\begin{proposition}\label{V2_equiv}
	$\V2 \uequiv \V3 \uequiv \V4 \uequiv \mathsf{sTC}_{\mathbb{N}}$.
\end{proposition}

\begin{proof}
	As already mentioned, $\V2 \ured \V3 \ured \V4$, so we have only to show that $\V4 \ured \mathsf{sTC}_{\mathbb{N}} \ured \V2$. For the first reduction, fix $f : 2^{<\omega} \to 2$ and define $A = \set{\sigma : (\forall \tau \succeq \sigma)[f(\tau) = 1]}$. We can regard $A$ as an instance of $\mathsf{sTC}_{\mathbb{N}}$, uniformly computable from $f$. Then $A \neq \emptyset$ if and only if $\set{\tau : f(\tau) = 0}$ is not dense. Given $-1$ as an $\mathsf{sTC}_{\mathbb{N}}$-solution to $A$, we can then take $\seq{0,\emptystring}$ as a $\V4$-solution to $f$. And given some $\mathsf{sTC}_{\mathbb{N}}$-solution to $A$ different from $-1$, coding a string $\sigma \in 2^{<\omega}$, we can take $\seq{1,\sigma}$ as a $\V4$-solution to~$f$.
	
	Now, to show that $\mathsf{sTC}_{\mathbb{N}} \ured \V2$, fix an instance $A$ of $\mathsf{sTC}_{\mathbb{N}}$, specified by a co-enumeration $e : \omega \to \omega$. Let $\ell$ be the function defined in the proof of \Cref{V1_equiv}. Define $f : 2^{<\omega} \to 2$ as follows: for every $s$, let $f(0^s) = 0$; for every $s$ and every $\sigma \succeq 0^s1$, let
	\[
		f(\sigma) =
		\begin{cases}
			0 & \text{if } (\exists t)[s \leq t < |\sigma| \wedge \ell(t) \neq \ell(t+1)],\\
			1 & \text{otherwise}.
		\end{cases}
	\]
	Note that $A \neq \emptyset$ if and only if $\lim_s \ell(s)$ exists (and belongs to $A$), if and only if there is a $\sigma$ such that $f(\tau) = 1$ for all $\tau \succeq \sigma$. So suppose $\seq{i,\sigma}$ is a $\V2$-solution to $f$. If $i = 0$, then $\set{\tau : f(\tau) = 0}$ is dense, hence by the previous observation we must have $A = \emptyset$. If $i = 1$, then we claim that $\ell(|\sigma|) \in A$. Suppose not, so that for some $t \geq |\sigma|$ we have $\ell(t) \neq \ell(t+1)$. Fix $\sigma' \succeq \sigma1$ with $|\sigma'| > t$. Then $\sigma' \succeq 0^s1$ for some $s \leq t$, hence $s \leq t < |\sigma'|$. By definition of $f$, we have that $f(\tau) = 0$ for all $\tau \succeq \sigma'$. But since $\seq{1,\sigma}$ is a $\V2$-solution to $f$, we also have that $\set{\tau : f(\tau) = 1}$ is dense above $\sigma$, a contradiction. Thus, from any $\V2$-solution to $f$ we can uniformly computably determine whether or not $A$ is empty, and if it not, determine an element of $A$, which completes the proof.
\end{proof}

We are now ready to prove the main theorem of this section.

\begin{proof}[Proof of \Cref{thm_Vstuff_MAIN}]
	We have only to justify that $\V0 \nured \TT^1_2$ and that $\V2 \nured \V1$. The former is a direct consequence of \Cref{TT12_not_FO}, to be proved later, which establishes the stronger fact that $\TT^1_2$ is not Weihrauch equivalent to any first-order problem. The latter follows from \Cref{V1_equiv,V2_equiv}, and the fact that $\mathsf{sTC}_{\mathbb{N}} \nured \mathsf{TC}_{\mathbb{N}}$. To justify this last fact, consider for example the problem $\mathsf{isFinite}$, whose instances are characteristic functions of subsets of $\omega$ having, as their unique solution, either $1$ or $0$ depending as this subset is finite or infinite. Clearly, $\mathsf{isFinite} \ured \mathsf{sTC}_{\mathbb{N}}$ since the set of upper bounds of a given subset of $\omega$ is $\mathbf{\Pi}^0_1$. But $\mathsf{isFinite} \nured \mathsf{TC}_{\mathbb{N}}$ by a result of Neumann and Pauly~\cite[Proposition 24]{NP-2018}.
\end{proof}

\section{Bounded number of colors}\label{sec:boundedcols}

In this section, we analyze the strength of $\TT^1_k$ for fixed $k \in \omega$, as well as of $\TT^1_+$, which is $\TT^1$ restricted to colorings for which an upper bound on the number of colors is known explicitly.

As noted in \Cref{sec:intro,sec:background}, $\RT^1_k \ured \TT^1_k$. It turns out that this is strict, in a strong sense that we now explain. First, we can show that $\RT^1_k$ cannot be obtained from $\TT^1$ using fewer than $k$ many colors. This strengthens the result that $\RT^1_k \nured \RT^1_j$ if $k > j$, due independently to Brattka and Rakotoniaina~\cite[Theorem 4.22]{BR-2017}, Dorais, et al.~\cite[Theorem 3.1]{DDHMS-2016}, and Hirschfeldt and Jockusch~\cite[Theorem 3.4]{HJ-2016}. We thank the anonymous referee for pointing out that this result also follows from our \Cref{thm:TT1k_not_D2k} below and the fact that $\RT^1_k \nured \D^2_j$ for all $k > j$ (see \cite{HJ-2016}, Theorem 2.10). We include a direct proof here for completeness.

\begin{theorem}\label{RT1more_non_TT1less}
	For every $k > j$, $\RT^1_k \nured \TT^1_j$.	
\end{theorem}

\begin{proof}
	The proof is quite similar to the proof that $\RT^1_k \nured \RT^1_j$. We spell out the details. Suppose $\RT^1_k \ured \TT^1_j$ via $\Phi$ and $\Psi$. We define a sequence $\alpha_0 \preceq \cdots \preceq \alpha_j$ of elements of $k^{<\omega}$, regarded as initial segments of an instance of $\RT^1_k$, along with a sequence of strings $\sigma_0 \preceq \cdots \preceq \sigma_j$ of $2^{<\omega}$.
	
	Let $\alpha_0 = \emptystring$ and $\sigma_0 = \emptystring$. Now fix $c < j$ and suppose $\alpha_c$ and $\sigma_c$ have been defined. Search for numbers $m,n \in \omega$ and a finite set $S$ satisfying the following:
	\begin{itemize}
		\item $S \cong 2^{< n}$, and the root of $S$ extends $\sigma_c$;
		\item $\Phi(\alpha_c c^m)(\sigma) \converges$ for all $\sigma \in S$ and has the same value;
		\item $\Psi(\alpha_c c^m,S)(x) \converges = 1$ for some $x \geq |\alpha_c|$.
	\end{itemize}
	The search must succeed, because $g = \alpha_c c^\omega$ is an instance of $\RT^1_k$, so $\Phi(g)$ is in turn an instance $f$ of $\TT^1_j$. Moreover, for any $H \cong 2^{<\omega}$ monochromatic for $f$ with root extending $\sigma_c$ we must have that $\Psi(g,H)$ is an infinite monochromatic set for $g$. In this case, any sufficiently large initial segment of $g$ can serve as $\alpha_c c^m$, and any sufficiently large initial segment of $H$ can serve as $S$.
	
	Having found $m$ and $n$, set $\alpha'_{c} = \alpha_c c^m$. Without loss of generality, we may assume $m+|\alpha_c| > x$ for the $x$ witnessing the third property above, so $\alpha'_{c}(x)$ is defined and equal to $c$. Now, call $\alpha \in k^{<\omega}$ \emph{$c$-good} if $\alpha \succeq \alpha_c'$ and $\alpha(x) > c$ for all $x \geq |\alpha_c'|$. Since $c < j < k$, $c$-good strings exist. Similarly, call $g : \omega \to k$ \emph{$c$-good} if $g \succ \alpha_c'$ and $g(x) > c$ for all $x \geq |\alpha_c'|$. If $g$ is $c$-good, then $\Phi(g)(\sigma) \converges$ for all $\sigma \in S$ and has the same value, and $\Psi(g,S)(x) \converges = 1$ for some $x$ with $g(x) = c$. But $g$, being $c$-good, has no infinite monochromatic set of color $c$, so $S$ cannot be extended to any $H \cong 2^{<\omega}$ monochromatic for $\Phi(g)$. In particular, it cannot be that for every leaf $\lambda$ of $S$ the set $\set{\sigma : \Phi(g)(\sigma) = \Phi(g)(\lambda)}$ is dense above $\lambda$. It follows that there exists a $c$-good $\alpha$, a leaf $\lambda$ of $S$, and a string $\sigma \succeq \lambda$ such that for all $c$-good $\beta \succeq \alpha$ and all $\tau \succeq \sigma$, if $\Phi(\beta)(\tau) \converges$ then $\Phi(\beta)(\tau) \neq \Phi(\beta)(\lambda)$. Fix some such $\alpha$ and $\sigma$, and let $\alpha_{c+1} = \alpha$ and $\sigma_{c+1} = \sigma$.
	
	Now, fix any $g : \omega \to k$ extending $\alpha_j$ which is $c$-good for all $c < j$. Let $f = \Phi(g)$. By assumption, $f$ is a $j$-coloring. But by induction, $|\set{f(\tau) : \tau \succeq \sigma_c}| \leq j-c$ for all $c \leq j$. In particular, this means $|\set{f(\tau) : \tau \succeq \sigma_j}| = 0$, which is impossible.
\end{proof}

Since, as noted in the introduction, $\RT^1_k \ured \fopart \TT^1_k \ured \TT^1_k$ for all $k$, we immediately get the following corollary.

\begin{corollary}
	For all $k > j$, $\TT^1_k \nured \TT^1_j$ and $\fopart \TT^1_k \nured \fopart \TT^1_j$.
\end{corollary}

In the opposite direction, the next result shows that $\TT^1_k$ cannot be obtained from $\RT^1$ no matter how many colors the latter is allowed to use. The proof is considerably more subtle and will serve as a basis for the remaining major theorems of this section. We first prove a lemma.

\begin{lemma}\label{lem:incomp_strings}
	Fix $n \geq 1$. Let $S_i \subseteq 2^{< \omega}$, $i < n$, be pairwise disjoint sets of incomparable strings with $|S_i| \geq n$. There is a set of incomparable strings $\{ \sigma_i : i < n \}$ such that $\sigma_i \in S_i$ for each $i < n$.
\end{lemma}

\begin{proof}
	The proof is by induction on $n$. For $n=1$, the result is trivial, so assume $n > 1$. Let $S = \cup_{i < n} S_i$. Fix $\sigma \in S$ such that $\text{rk}_S(\sigma)$ is maximal. By reindexing the $S_i$ sets if necessary, we can assume that $\sigma \in S_{n-1}$. Let 
	\[
		S' = S \setminus \{ \tau \in S : \tau \text{ and } \sigma \text{ are comparable} \} \text{ and } S_i' = S_i \cap S' \text{ for } i < n-1.
	\]
	Because $\text{rk}_S(\sigma)$ is maximal, there are no strings $\tau \in S$ such that $\sigma \prec \tau$. Since the sets $S_i$ consist of incomparable strings, it follows that $|S_i \setminus S_i'| \leq 1$, and hence $|S_i'| \geq n-1$, for each $i < n-1$. Therefore, we can apply the inductive hypothesis to choose a set of incomparable strings $\sigma_i \in S'_i$ for $i < n-1$.  The set $\{ \sigma_0, \ldots, \sigma_{n-2}, \sigma \}$ is the desired set of incomparable strings. 
\end{proof}

The following result may be regarded as an analogue of Lemma 3.8 of Corduan, Groszek, and Mileti~\cite{CGM-2010}, which establishes that $\RCA_0 + \mathsf{B}\Sigma^0_2$ does not prove $\TT^1$.

\begin{theorem}\label{1TTnuredRT1}
	For every $k \geq 2$ and every $j \geq 1$, $\fopart \TT^1_k \nured \RT^1_j$.
\end{theorem}

\begin{proof}
	It suffices to prove the result for $k = 2$. Seeking a contradiction, suppose $\fopart \TT^1_2 \ured \RT^1_j$ via $\Phi$ and $\Psi$. Let $\Gamma$ be a functional such that, given $f : 2^{<\omega} \to 2$ and $S \cong 2^{<\omega}$, $\Gamma(f,S)(0)$ outputs $\seq{\sigma_0,\ldots,\sigma_{j-1}}$ for the least $j$ many pairwise incomparable strings~ $\sigma_0,\ldots,\sigma_{j-1} \in S$. In particular, for every $f : 2^{<\omega} \to 2$, $\seq{f,\mathrm{Id},\Gamma}$ is an instance of $\fopart \TT^1_2$. Hence by assumption, $\Phi(f,\mathrm{Id},\Gamma)$ is an instance $g$ of $\RT^1_j$. Moreover, note that for each $c < j$, either $g(x) \neq c$ for all sufficiently large $x$, or there is a finite set $E$ monochromatic for $g$ with color $c$ such that $\Psi(f,\mathrm{Id},\Gamma,E)(0) \converges = \seq{\sigma_0,\ldots,\sigma_{j-1}}$ for some pairwise incomparable strings $\sigma_0,\ldots,\sigma_{j-1}$.
	
	Consider Cohen forcing, with conditions $\alpha \in 2^{<\omega}$ regarded as initial segments of a $2$-coloring of $2^{<\omega}$. Let $\check{f}$ be a name in the forcing language for a generic such $2$-coloring, and let $\check{g}$ be a name for $\Phi(\check{f},\mathrm{Id},\Gamma)$. From the above discussion, it follows that we can find a condition $\alpha$ along with a number $n_0$, a non-empty set $C \subseteq j$ and, for each $c \in C$, a finite set $E_c$, such that $\alpha$ forces each of the following:
	\begin{itemize}
		\item $E_c$ is monochromatic for $\check{g}$ with color $c$;
		\item $\Psi(\check{f},\mathrm{Id},\Gamma,E_c)(0) \converges$;
		\item $\check{g}(x) \in C$ for all $x > n_0$.
	\end{itemize}
	Notice that the first two clauses are $\Sigma^0_1$, while the third is $\Pi^0_1$. Therefore, if $f : 2^{<\omega} \to 2$ is any coloring extending $\alpha$, whether generic or not, each of the above clauses holds if $\check{f}$ is replaced by $f$ and $\check{g}$ by $g = \Phi(f,\mathrm{Id},\Gamma)$. 
	
	By definition of $\Gamma$, for each $c \in C$ we have $\Psi(\alpha,\mathrm{Id},\Gamma,E_c)(0) = \seq{\sigma_{c,0},\ldots,\sigma_{c,j-1}}$ for some pairwise incomparable strings $\sigma_{c,0},\ldots,\sigma_{c,j-1} \in 2^{<\omega}$. By standard use conventions, we may assume $\sigma_{c,0},\ldots,\sigma_{c,j-1} < |\alpha|$ for all $c$. Apply \Cref{lem:incomp_strings} to choose $\sigma_c \in \set{\sigma_{c,0},\ldots,\sigma_{c,j-1}}$ for each $c \in C$ such that $\set{\sigma_c : c \in C}$ is a set of incomparable strings.
		
	Finally, we are ready to define $f$. For all $\sigma < |\alpha|$, let $f(\sigma) = \alpha(\sigma)$. As noted above, this defines $f(\sigma_c) = \alpha(\sigma_c)$ for all $c \in C$. If $\sigma \geq |\alpha|$ and $\sigma \nsucceq \sigma_c$ for all $c \in C$, let $f(\sigma) = 0$. If $\sigma \geq |\alpha|$ and $\sigma \succeq \sigma_c$ for some $c \in C$, then $\sigma \nsucceq \sigma_d$ for all $d \neq c$ in $C$ since $\sigma_d$ and $\sigma_c$ are incomparable. In this case, we let $f(\sigma) = 1 - f(\sigma_c)$. This completes the definition. To complete the proof, let $g = \Phi(f,\mathrm{Id},\Gamma)$. Since $f$ extends $\alpha$ it follows by our earlier remark that each $E_c$ is monochromatic for $g$ with color $c$, and every infinite monochromatic set for $g$ has color some $c \in C$. It follows that for some $c \in C$, $E_c$ is extendible to an $\RT^1_j$-solution to $g$, and therefore $\seq{\sigma_{c,0},\ldots,\sigma_{c,j-1}}$ is a $\fopart \TT^1_2$-solution to $\seq{f,\mathrm{Id},\Gamma}$. This means that $\seq{\sigma_{c,0},\ldots,\sigma_{c,j-1}} = \Gamma(f,H)$ for some $H \cong 2^{<\omega}$ monochromatic for~$f$. But for every $\sigma \succeq \sigma_c$ with $|\sigma| \geq \alpha$, and in particular, for every such $\sigma \in H$, we have that $f(\sigma) \neq f(\sigma_c)$. Thus, $H$ cannot be monochromatic for $f$, a contradiction.
\end{proof}

\begin{corollary}\label{cor:1TTnuredRT1}
	For every $k \geq 2$ and every $j \geq 1$, $\TT^1_k \nured \RT^1_j$.
\end{corollary}

At first blush, it may seem that $\RT^1_j$ in the statement of \Cref{1TTnuredRT1} could be replaced by $\RT^1_+$. After all, an instance of $\RT^1_+$ is an instance of $\RT^1_j$ for some $j$, and this $j$ is specified. But there is a problem with lifting the proof directly. Namely, the instance of $\fopart \TT^1_k$ constructed in the proof requires both an instance $f$ of $\TT^1_k$ and a functional $\Gamma$, and this functional very much depends on $j$. The number $j$ is, in turn, computed via a hypothetical reduction $\Phi$ from both $f$ and $\Gamma$, which results in a circularity. The natural attempt to get around this is to use the recursion theorem, and this works, but at a cost: the argument only works to show that $\fopart \TT^1_k \nured \RT^1_+$ for $k \geq 3$.

\begin{theorem}\label{thm:TT1k_not_to_RT1+}
	For every $k \geq 3$, $\fopart \TT^1_k \nured \RT^1_+$.	
\end{theorem}

\begin{proof}
	It suffices to prove the result for $k = 3$. Suppose to the contrary that $\fopart \TT^1_3 \ured \RT^1_+$, say via $\Phi$ and $\Psi$. We build a $3$-coloring $f$ of $2^{<\omega}$ and a Turing functional $\Gamma$ so that $\seq{f,\mathrm{Id},\Gamma}$ is an instance of $\fopart \TT^1_3$ witnessing a contradiction. The way we define $\Gamma$ is as follows. Let $\Delta_0,\Delta_1,\ldots$ be the standard enumeration of all oracle Turing functionals (so named in order to avoid notationally clashing with the fixed functional $\Phi$). We define a certain computable function $h$ and then choose a fixed point $e_0$ for $h$ using the recursion theorem, so that $\Delta_{h(e_0)}(A) = \Delta_{e_0}(A)$ for all oracles $A$. We then let $\Gamma = \Delta_{e_0}$.
	
	For the definition of $h$, we decompose each oracle $A$ as $\seq{f,S}$, with our interest being in the situation where $f$ is a $3$-coloring of $2^{<\omega}$ and $S$ is a monochromatic set for $f$. Of course, for an arbitrary $A$, either $f$ or $S$ or both may fail to have these properties, but we do not care about these cases. We obtain $h$ using the $s$-$m$-$n$ theorem so as to satisfy the following properties, for all $e \in \omega$ and all oracles $A = \seq{f,S}$. First, find the least $\sigma \in S$. If $f(\sigma) = 0$, then let $\Delta_{h(e)}(A)(0) \converges = \sigma$. If $f(\sigma) > 0$, then wait for $\Phi(f,\mathrm{Id},\Delta_e)(0)$ to converge and output a number $j \geq 1$. (The idea is that if $\seq{f,\mathrm{Id},\Delta_e}$ is an instance of $\fopart \TT^1_3$, then $\Phi(f,\mathrm{Id},\Delta_e)$ is an instance $\seq{j,g}$ of $\RT^1_+$, so $\Phi(f,\mathrm{Id},\Delta_e)(0) = j$.) Then, find the first $j$ many incomparable strings $\sigma_0,\ldots,\sigma_{j-1} \in S$ and let $\Delta_{h(e)}(A)(0) \converges = \seq{\sigma_0,\ldots,\sigma_{j-1}}$. 
	
	Let $\Gamma$ be as described above, and consider the trivial coloring $f_0 : 2^{<\omega} \to 3$ such that $f_0(\sigma) = 0$ for all $\sigma \in 2^{<\omega}$. If $H \cong 2^{<\omega}$ is monochromatic for $f_0$ then it is clear from the definition that $\Gamma(f_0,H)(0) \converges$. (Namely, $\Gamma(f_0,H)(0)$ outputs the least string in $H$.) Thus, $\seq{f_0,\mathrm{Id},\Gamma}$ is an instance of $\fopart \TT^1_3$. By assumption, $\Phi(f_0,\mathrm{Id},\Gamma)(0) \converges = j$ for some $j \geq 1$. Fix a finite initial segment $\alpha_0 \prec f_0$ such that $\Phi(\alpha_0,\mathrm{Id},\Gamma)(0) \converges = j$. Say $f : 2^{<\omega} \to 3$ is \emph{good} if $f$ extends $\alpha_0$ and $f(\sigma) \neq 0$ for all $\sigma \geq |\alpha_0|$. Similarly, say $\alpha \in 3^{<\omega}$, regarded as a $3$-coloring of a finite subset of $2^{<\omega}$, is \emph{good} if $\alpha_0 \preceq \alpha$ and $\alpha(\sigma) \neq 0$ for all $\sigma \geq |\alpha_0|$.
	
	Now, consider any good coloring $f : 2^{<\omega} \to 3$. Then $\seq{f,\mathrm{Id},\Gamma}$ is an instance of $\fopart \TT^1_3$. Indeed, we have that $\Phi(f,\mathrm{Id},\Gamma)(0) \converges = j$ since $f$ extends $\alpha_0$. Moreover, any $H \cong 2^{<\omega}$ monochromatic for $f$ must have color $1$ or $2$. Thus, by definition of $\Gamma$, we have that $\Gamma(f,H)(0) \downarrow = \seq{\sigma_0,\ldots,\sigma_{j-1}}$ for the least $j$ many incomparable strings $\sigma_0,\ldots,\sigma_{j-1} \in H$. And $\Phi(f,\mathrm{Id},\Gamma) = \seq{j,g}$ for some instance $g$ of $\RT^1_j$. So we can now proceed precisely as in the proof of \Cref{1TTnuredRT1}, only working with good conditions in $3^{<\omega}$, and building a good $3$-coloring $f$.
\end{proof}

Remarkably, after a preprint of this paper was circulated, Pauly \cite{Pauly-TA} has shown that $\fopart \TT^1_2 \ured \RT^1_+$. Thus, \Cref{thm:TT1k_not_to_RT1+} is optimal. 

We now move from $\fopart \TT^1_k$ to $\TT^1_k$, and from $\RT^1_+$ to $\RT^1_{\mathbb{N}}$. The obstacle with needing to know a functional $\Gamma$ ``in advance'' does not arise here (since we do not need to build an instance of a first-order part), so we get a result that holds for all $k \geq 2$.

\begin{theorem}\label{thm:TT1k_not_to_RT1N}
	For every $k \geq 2$, $\TT^1_k \nured \RT^1_{\mathbb{N}}$.	
\end{theorem}

\begin{proof}
	Again, the proof is based on that of \Cref{1TTnuredRT1}, but with some small differences. As there, we can take $k = 2$, and then argue by contradiction. So, asume that $\TT^1_2 \ured \RT^1_{\mathbb{N}}$, via $\Phi$ and $\Psi$. Consider Cohen forcing in $2^{<\omega}$, let $\check{f}$ be as in \Cref{1TTnuredRT1}, and let $\check{g}$ be a name in the forcing language for $\Phi(\check{f})$. Since a generic here is an instance of $\TT^1_2$, which is in turn mapped by $\Phi$ to an instance of $\RT^1_{\mathbb{N}}$, it follows that we can fix a $j \geq 1$ and a condition $\alpha_0$ forcing that $\check{g}$ is a $j$-coloring. This can be expressed as a $\Pi^0_1$ property, hence it will be preserved by any $2$-coloring of $2^{<\omega}$ that extends $\alpha_0$, generic or not. We can now proceed basically just like in \Cref{1TTnuredRT1}, but working below $\alpha_0$. More precisely, we find a condition $\alpha \geq \alpha_0$ along with a number $n_0$, a non-empty set $C \subseteq j$ and, for each $c \in C$, a finite set $E_c$, such that $\alpha$ forces each of the following:
	\begin{itemize}
		\item $E_c$ is monochromatic for $\check{g}$ with color $c$;
		\item $\Psi(\check{f},E_c)(\sigma) \converges = 1$ for $j$ many pairwise incomparable strings $\sigma$;
		\item $\check{g}(x) \in C$ for all $x > n_0$.
	\end{itemize}
	The rest of the construction of a $2$-coloring $f \succeq \alpha$ of $2^{<\omega}$ is now just as before.
\end{proof}

\Cref{thm:TT1k_not_to_RT1+,thm:TT1k_not_to_RT1N} raise two natural questions. The first is what happens between $\fopart \TT^1_k$ and $\RT^1_{\mathbb{N}}$, and the second is what happens between $\fopart \TT^1_k$ and $\TT^1_k$. These questions will be answered in the next section, in \Cref{cor:TT1N_equiv_RT1N,TT12_not_FO}.

We conclude by looking at the relationship of $\TT^1_k$ and $\D^2_k$, a close relative of $\RT^1_k$. Recall that $\D^2_k$ is the problem whose instances are all $\mathbf{\Delta}^0_2$ colorings $f : \omega \to k$, with the solution to any such $f$ being all infinite set $M \subseteq \omega$ homogeneous for $f$. This problem, formulated as a statement of second-order arithmetic, was introduced by Cholak, Jockusch, and Slaman~\cite[Statement 7.8]{CJS-2001}, and has played an important role in the reverse mathematical analysis of Ramsey's theorem for pairs. (See~\cite[Sections 8]{DM-2022} for more of the history.)

Often, the instances of $\D^2_k$ are represented using Schoenfield's limit lemma, as colorings $f : \omega^2 \to k$ with the property that $\lim_y f(x,y)$ exists for all $x$. Here, we think of them as pairs $(X,f)$ where $X$ is an arbitrary set and $f$ is an $X'$-computable instance of $\RT^1_k$. In jargon, $\D^2_k$ is $(\RT^1_k)'$, where $'$ is the Weihrauch jump operator (cf.~\cite[Definitions 6.8 and 6.9]{BGP-2021}). The following theorem may thus be compared against the previous one.

\begin{theorem}\label{thm:TT1k_not_D2k}
	For every $k \geq 1$, $\TT^1_k \ured \D^2_k$.	
\end{theorem}

\begin{proof}
	The result is trivial for $k = 1$, so we may assume $k \geq 2$. Fix an instance $f : 2^{<\omega} \to k$ of $\TT^1_k$. For each $n \in \omega$ and each $c < k-1$ we define a string $\rho_{c,n} \in 2^{<\omega}$ by induction on $c$. Let $\rho_{-1,n} = \emptystring$ for definiteness, and assume that $\rho_{c-1,n}$ has been defined for some $c < k-1$. Then, let $\rho_{c,n}$ be the least extension of $\rho_{c-1,n}$ (in some fixed enumeration of $2^{<\omega}$)  such that $f(\sigma) \neq c$ for all $\sigma \succeq \rho_{c,n}$ with $|\sigma| < n$. Notice that $\rho_{c,n}$ is always defined, and that it is uniformly computable from $c$ and $n$. If $\lim_n \rho_{c,n}$ exists, we denote the limit by $\rho_c$.
	
	We next define a coloring $g : \omega \to k$. Given $n \in \omega$, choose the largest $c < k$ such that $\rho_{d,n} = \rho_{d,m}$ for all $d < c$ and all $m \geq n$. (There is such a $c$ because, in particular, $\rho_{-1,n} = \rho_{-1,m}$ for all~$m \geq n$.) Then, let $g(n) = c$. Observe that $g$ is uniformly computable from $f'$, and so it is an instance of $\D^2_k$.
	
	Fix the largest $c < k$ such that $\rho_d$ is defined for all $d < c$. (Again, there is such a $c$ because $\rho_{-1}$ is defined.) Then by definition of $g$, for all sufficiently large $n$ we have that $g(n) = c$. Moreover, the set $\set{\sigma : f(\sigma) = c}$ must be dense above $\rho_{c-1}$. Suppose not. Then we may fix the least $\rho \succeq \rho_{c-1}$ such that $f(\sigma) \neq c$ for all $\sigma \succeq \rho$. Since $\rho_{-1,n} \preceq \cdots \preceq \rho_{c-1,n}$ for all $n$, we have $\rho_{-1} \preceq \cdots \preceq \rho_{c-1}$. Hence, by construction, $f(\sigma) > c$ for all $\sigma \succeq \rho$. It follows that $c < k-1$ and that $\rho_{c,n} = \rho$ for all sufficiently large $n$. But then $\rho_c$ is defined, a contradiction. 
	
	To complete the proof, let $M$ be any infinite monochromatic set for $g$, and let $n = \min M$. By the argument above, the color of $M$ under $g$ must be $c$, so also $g(n) = c$. By definition of $g$, this means that $\rho_{c-1,n} = \rho_{c-1,m}$ for all $m \geq n$, hence $\rho_{c-1,n} = \rho_{c-1}$. Since $\set{\sigma : f(\sigma) = c}$ is dense above $\rho_{c-1}$, we can uniformly compute from $f \oplus M$ an $H \cong 2^{<\omega}$ monochromatic for $f$ with color $c$ and root extending~$\rho_{c-1}$.
\end{proof}

\section{Unbounded number of colors}\label{sec:rakes}

We now turn to studying the problem $\TT^1_{\mathbb{N}}$. As mentioned in the introduction, this is arguably the formulation closest to the principle $\TT^1$ as it is studied in computability theory and reverse mathematics. Our main theorem here is \Cref{TT1N_theorem_main}, that $\fopart \TT^1_{\mathbb{N}} \ured \RT^1_{\mathbb{N}}$. The proof requires some new combinatorial machinery which we first carefully develop.

\begin{definition}
	Fix a coloring $f$ of $2^{<\omega}$, $k \geq 1$, and $C = \set{c_0 < \cdots < c_{k-1}} \subseteq \omega$.
	\begin{enumerate}
		\item A \emph{$C$-rake for $f$} is a set $R \subseteq 2^{<\omega}$ with the following properties.
		\begin{itemize}
			\item $R$ is rooted.
			\item $R$ is symmetric.
			\item If $\rank_R(\sigma) \equiv i \mod k$ then $f(\sigma) = c_i$.
			\item If $\rank_R(\sigma) \not\equiv k-1 \mod k$ then $\sigma$ has exactly one successor in $R$.
			\item If $\rank_R(\sigma) \equiv k-1 \mod k$ then $\sigma$ has $0$ or $k+1$ many successors in $R$.
		\end{itemize}
		\item The \emph{block} of $\sigma \in R$ in $R$, denoted $\block_R(\sigma)$, is $\lfloor \frac{\rank_R(\sigma)}{k} \rfloor$.
		\item $R$ is \emph{good} if $f(\sigma) \in C$ for all~$\sigma \in 2^{<\omega}$ extending the root of $R$.
	\end{enumerate}
\end{definition}
\noindent See \Cref{fig:C-rake} for a typical example.

\begin{figure}[htbp]
  \centering
	\hspace{60pt}\begin{tikzpicture}[scale=0.8,font=\normalsize]
		\tikzset{
		empty node/.style={circle,inner sep=0,fill=none},
		solid node/.style={circle,draw,inner sep=1.5,fill=black},
		hollow node/.style={circle,draw,thick,inner sep=1.5,fill=white}
		}
		\node(0)[hollow node] at (0,0) {\tiny 0};
		\node(1)[hollow node] at (0,1) {\tiny 1};
		\node(2)[hollow node] at (0,2) {\tiny 2};
		
		\node(a0)[hollow node] at (-3,3) {\tiny 0};
		\node(b0)[hollow node] at (-1,3) {\tiny 0};
		\node(c0)[hollow node] at (1,3) {\tiny 0};
		\node(d0)[hollow node] at (3,3) {\tiny 0};
		
		\node(a1)[hollow node] at (-3,4) {\tiny 1};
		\node(b1)[hollow node] at (-1,4) {\tiny 1};
		\node(c1)[hollow node] at (1,4) {\tiny 1};
		\node(d1)[hollow node] at (3,4) {\tiny 1};
		
		\node(a2)[hollow node] at (-3,5) {\tiny 2};
		\node(b2)[hollow node] at (-1,5) {\tiny 2};
		\node(c2)[hollow node] at (1,5) {\tiny 2};
		\node(d2)[hollow node] at (3,5) {\tiny 2};
		
		\node(aa0)[hollow node] at (-3.75,6) {};
		\node(ab0)[hollow node] at (-3.25,6) {};
		\node(ac0)[hollow node] at (-2.75,6) {};
		\node(ad0)[hollow node] at (-2.25,6) {};
		
		\node(ba0)[hollow node] at (-1.75,6) {};
		\node(bb0)[hollow node] at (-1.25,6) {};
		\node(bc0)[hollow node] at (-0.75,6) {};
		\node(bd0)[hollow node] at (-0.25,6) {};
		
		\node(ca0)[hollow node] at (0.25,6) {};
		\node(cb0)[hollow node] at (0.75,6) {};
		\node(cc0)[hollow node] at (1.25,6) {};
		\node(cd0)[hollow node] at (1.75,6) {};
		
		\node(da0)[hollow node] at (2.25,6) {};
		\node(db0)[hollow node] at (2.75,6) {};
		\node(dc0)[hollow node] at (3.25,6) {};
		\node(dd0)[hollow node] at (3.75,6) {};
		
		\node(aa1)[hollow node] at (-3.75,7) {};
		\node(ab1)[hollow node] at (-3.25,7) {};
		\node(ac1)[hollow node] at (-2.75,7) {};
		\node(ad1)[hollow node] at (-2.25,7) {};
		
		\node(ba1)[hollow node] at (-1.75,7) {};
		\node(bb1)[hollow node] at (-1.25,7) {};
		\node(bc1)[hollow node] at (-0.75,7) {};
		\node(bd1)[hollow node] at (-0.25,7) {};
		
		\node(ca1)[hollow node] at (0.25,7) {};
		\node(cb1)[hollow node] at (0.75,7) {};
		\node(cc1)[hollow node] at (1.25,7) {};
		\node(cd1)[hollow node] at (1.75,7) {};
		
		\node(da1)[hollow node] at (2.25,7) {};
		\node(db1)[hollow node] at (2.75,7) {};
		\node(dc1)[hollow node] at (3.25,7) {};
		\node(dd1)[hollow node] at (3.75,7) {};
		
		\node(aa2)[hollow node] at (-3.75,8) {};
		\node(ab2)[hollow node] at (-3.25,8) {};
		\node(ac2)[hollow node] at (-2.75,8) {};
		\node(ad2)[hollow node] at (-2.25,8) {};
		
		\node(ba2)[hollow node] at (-1.75,8) {};
		\node(bb2)[hollow node] at (-1.25,8) {};
		\node(bc2)[hollow node] at (-0.75,8) {};
		\node(bd2)[hollow node] at (-0.25,8) {};
		
		\node(ca2)[hollow node] at (0.25,8) {};
		\node(cb2)[hollow node] at (0.75,8) {};
		\node(cc2)[hollow node] at (1.25,8) {};
		\node(cd2)[hollow node] at (1.75,8) {};
		
		\node(da2)[hollow node] at (2.25,8) {};
		\node(db2)[hollow node] at (2.75,8) {};
		\node(dc2)[hollow node] at (3.25,8) {};
		\node(dd2)[hollow node] at (3.75,8) {};
		
		\draw[-,thick] (0) to (1) to (2);
		
		\draw[-,thick] (2) to (a0);
		\draw[-,thick] (2) to (b0);
		\draw[-,thick] (2) to (c0);
		\draw[-,thick] (2) to (d0);
		
		\draw[-,thick] (a0) to (a1) to (a2);
		\draw[-,thick] (b0) to (b1) to (b2);
		\draw[-,thick] (c0) to (c1) to (c2);
		\draw[-,thick] (d0) to (d1) to (d2);
		
		\draw[-,thick] (a2) to (aa0);
		\draw[-,thick] (a2) to (ab0);
		\draw[-,thick] (a2) to (ac0);
		\draw[-,thick] (a2) to (ad0);
		
		\draw[-,thick] (b2) to (ba0);
		\draw[-,thick] (b2) to (bb0);
		\draw[-,thick] (b2) to (bc0);
		\draw[-,thick] (b2) to (bd0);
		
		\draw[-,thick] (c2) to (ca0);
		\draw[-,thick] (c2) to (cb0);
		\draw[-,thick] (c2) to (cc0);
		\draw[-,thick] (c2) to (cd0);
		
		\draw[-,thick] (d2) to (da0);
		\draw[-,thick] (d2) to (db0);
		\draw[-,thick] (d2) to (dc0);
		\draw[-,thick] (d2) to (dd0);
		
		\draw[-,thick] (aa0) to (aa1) to (aa2);
		\draw[-,thick] (ab0) to (ab1) to (ab2);
		\draw[-,thick] (ac0) to (ac1) to (ac2);
		\draw[-,thick] (ad0) to (ad1) to (ad2);
		
		\draw[-,thick] (ba0) to (ba1) to (ba2);
		\draw[-,thick] (bb0) to (bb1) to (bb2);
		\draw[-,thick] (bc0) to (bc1) to (bc2);
		\draw[-,thick] (bd0) to (bd1) to (bd2);
		
		\draw[-,thick] (ca0) to (ca1) to (ca2);
		\draw[-,thick] (cb0) to (cb1) to (cb2);
		\draw[-,thick] (cc0) to (cc1) to (cc2);
		\draw[-,thick] (cd0) to (cd1) to (cd2);
		
		\draw[-,thick] (da0) to (da1) to (da2);
		\draw[-,thick] (db0) to (db1) to (db2);
		\draw[-,thick] (dc0) to (dc1) to (dc2);
		\draw[-,thick] (dd0) to (dd1) to (dd2);
		
		\draw [thick,decorate,decoration={calligraphic brace,amplitude=5pt,mirror,raise=110pt},yshift=0pt]
(0,0) to (0,2) node [black,midway,xshift=165pt,yshift=28pt] {\footnotesize
block $0$};

		\draw [thick,decorate,decoration={calligraphic brace,amplitude=5pt,mirror,raise=110pt},yshift=0pt]
(0,3) to (0,5) node [black,midway,xshift=165pt,yshift=114pt] {\footnotesize
block $1$};

		\draw [thick,decorate,decoration={calligraphic brace,amplitude=5pt,mirror,raise=110pt},yshift=0pt]
(0,6) to (0,8) node [black,midway,xshift=165pt,yshift=200pt] {\footnotesize
block $3$};
		
	\end{tikzpicture}
  	\caption{An illustration of a $C$-rake with $C = \set{0,1,2}$ of height $9$. The color under $f$ of the nodes in blocks $0$ and $1$ is indicated.}
  	\label{fig:C-rake}
\end{figure}
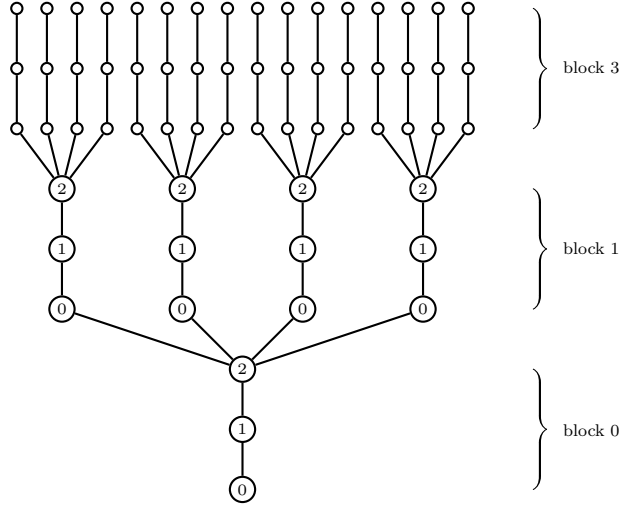

\begin{lemma}\label{exists_dense_set_of_colors}
	Fix a computable coloring $f$ of\, $2^{< \omega}$, a non-empty finite set $C \subseteq \omega$, and a string $\sigma \in 2^{<\omega}$. There is a set $W_{f,C,\sigma} \subseteq C$ such that for some $\tau \succeq \sigma$ the following hold:
	\begin{enumerate}
		\item $W_{f,C,\sigma}$ is uniformly $\Sigma^0_2$ in $\sigma$, a $\Delta^0_1$ index for $f$, and a canonical index for~$C$.
		\item $\set{\rho \in 2^{<\omega} : f(\rho) = c}$ is dense above~$\tau$ for each $c \in C \setminus W_{f,C,\sigma}$.
		\item $f(\rho) \notin W_{f,C,\sigma}$ for all $\rho \succeq \tau$.
	\end{enumerate}
\end{lemma}

\begin{proof}
	We use a $\emptyset'$ oracle to enumerate $W_{f,C,\sigma}$, letting $W_{f,C,\sigma}[s]$ denote the set of elements enumerated by stage $s$. In parallel, we define $\tau$ by finitely many finite extensions of $\sigma$. At stage $0$, we set $\tau_0 = \sigma$. At stage $s > 0$, we assume inductively that we have already defined $\tau_{s-1}$. We now check if there exists $\tau' < s$ extending $\tau_{s-1}$ such that for some $c \in C \setminus W_{f,C,\sigma}[s]$ we have that $f(\rho) \neq c$ for all $\rho \succeq \tau'$. If so, we let $\tau_s$ be the least such $\tau'$, and we enumerate all corresponding $c$ into $W_{f,C,\sigma}$. Otherwise, we let $\tau_s = \tau_{s-1}$. This completes the construction. We let $\tau = \lim_s \tau_s$. It is then easy to see that $W_{f,C,\sigma}$ has the desired properties, as witnessed by $\tau$.
\end{proof}

Observe that if $f$, $C$, and $\sigma$ above additionally satisfy that $f(\rho) \in C$ for all $\rho \succeq \sigma$, then by part (3) of the lemma it follows that $C \setminus W_{C,f,\sigma} \neq \emptyset$. This is in particular the case if we are dealing with a good $C$-rake $R$ for $f$ of finite height, and where $\sigma$ is an extension of some leaf of $R$. We will use this fact repeatedly in the sequel.

\begin{lemma}\label{exists_infinite_rake}
	Fix a computable coloring $f$ of $2^{<\omega}$. There exists a non-empty finite set $C \subseteq \omega$ and a computable good $C$-rake $R$ for $f$ of height $\omega$.
\end{lemma}

\begin{proof}
	Let $W_{f,\ran(f),\emptystring}$ be as given by \Cref{exists_dense_set_of_colors}, as witnessed by $\tau$. Define $C = \ran(f) \setminus W_{f,\ran(f),\emptystring}$. As noted above, $C$ is non-empty. By property (2), we can computably construct a $C$-rake $R$ for $f$ of height $\omega$ with root $\tau$. By property (3), $R$ is good.
\end{proof}

We now give a definition of a useful way to ``pick out'' monochromatic subtrees from rakes.

\begin{definition}
	Fix a coloring $f$ of $2^{<\omega}$, a non-empty finite set $C \subseteq \omega$, a $C$-rake $R$ for $f$ (of finite or infinite height), and a set $S \subseteq 2^{<\omega}$. We write $S \trianglelefteq R$ if the following properties hold.
	\begin{enumerate}
		\item $S \subseteq R$.
		\item $S \cong 2^{< n}$ for some $n \leq \height(R)$. (Note that we may have $n = \height(R) = \omega$.)
		\item $\rank_S(\sigma) = \block_R(\sigma)$ for all $\sigma \in S$.
		\item $\rank_R(\sigma) \equiv \rank_R(\tau) \mod |C|$ for all $\sigma,\tau \in S$.
		%\item For all $\sigma,\tau \in S$, $\rank_S(\sigma) = \rank_S(\tau)$ if and only if $\rank_R(\sigma) = \rank_R(\tau)$.
	\end{enumerate}
\end{definition}

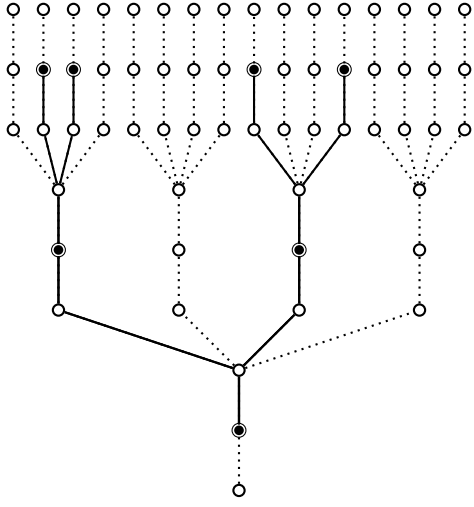
\begin{figure}[htbp]
  \centering
	\begin{tikzpicture}[scale=0.8,font=\normalsize]
		\tikzset{
		empty node/.style={circle,inner sep=0,fill=none},
		solid node/.style={circle,draw,inner sep=1.5,fill=black},
		hollow node/.style={circle,draw,thick,inner sep=1.5,fill=white}
		}
		\node(0)[hollow node] at (0,0) {};
		\node(1)[solid node,double] at (0,1) {};
		\node(2)[hollow node] at (0,2) {};
		
		\node(a0)[hollow node] at (-3,3) {};
		\node(b0)[hollow node] at (-1,3) {};
		\node(c0)[hollow node] at (1,3) {};
		\node(d0)[hollow node] at (3,3) {};
		
		\node(a1)[solid node,double] at (-3,4) {};
		\node(b1)[hollow node] at (-1,4) {};
		\node(c1)[solid node,double] at (1,4) {};
		\node(d1)[hollow node] at (3,4) {};
		
		\node(a2)[hollow node] at (-3,5) {};
		\node(b2)[hollow node] at (-1,5) {};
		\node(c2)[hollow node] at (1,5) {};
		\node(d2)[hollow node] at (3,5) {};
		
		\node(aa0)[hollow node] at (-3.75,6) {};
		\node(ab0)[hollow node] at (-3.25,6) {};
		\node(ac0)[hollow node] at (-2.75,6) {};
		\node(ad0)[hollow node] at (-2.25,6) {};
		
		\node(ba0)[hollow node] at (-1.75,6) {};
		\node(bb0)[hollow node] at (-1.25,6) {};
		\node(bc0)[hollow node] at (-0.75,6) {};
		\node(bd0)[hollow node] at (-0.25,6) {};
		
		\node(ca0)[hollow node] at (0.25,6) {};
		\node(cb0)[hollow node] at (0.75,6) {};
		\node(cc0)[hollow node] at (1.25,6) {};
		\node(cd0)[hollow node] at (1.75,6) {};
		
		\node(da0)[hollow node] at (2.25,6) {};
		\node(db0)[hollow node] at (2.75,6) {};
		\node(dc0)[hollow node] at (3.25,6) {};
		\node(dd0)[hollow node] at (3.75,6) {};
		
		\node(aa1)[hollow node] at (-3.75,7) {};
		\node(ab1)[solid node,double] at (-3.25,7) {};
		\node(ac1)[solid node,double] at (-2.75,7) {};
		\node(ad1)[hollow node] at (-2.25,7) {};
		
		\node(ba1)[hollow node] at (-1.75,7) {};
		\node(bb1)[hollow node] at (-1.25,7) {};
		\node(bc1)[hollow node] at (-0.75,7) {};
		\node(bd1)[hollow node] at (-0.25,7) {};
		
		\node(ca1)[solid node,double] at (0.25,7) {};
		\node(cb1)[hollow node] at (0.75,7) {};
		\node(cc1)[hollow node] at (1.25,7) {};
		\node(cd1)[solid node,double] at (1.75,7) {};
		
		\node(da1)[hollow node] at (2.25,7) {};
		\node(db1)[hollow node] at (2.75,7) {};
		\node(dc1)[hollow node] at (3.25,7) {};
		\node(dd1)[hollow node] at (3.75,7) {};
		
		\node(aa2)[hollow node] at (-3.75,8) {};
		\node(ab2)[hollow node] at (-3.25,8) {};
		\node(ac2)[hollow node] at (-2.75,8) {};
		\node(ad2)[hollow node] at (-2.25,8) {};
		
		\node(ba2)[hollow node] at (-1.75,8) {};
		\node(bb2)[hollow node] at (-1.25,8) {};
		\node(bc2)[hollow node] at (-0.75,8) {};
		\node(bd2)[hollow node] at (-0.25,8) {};
		
		\node(ca2)[hollow node] at (0.25,8) {};
		\node(cb2)[hollow node] at (0.75,8) {};
		\node(cc2)[hollow node] at (1.25,8) {};
		\node(cd2)[hollow node] at (1.75,8) {};
		
		\node(da2)[hollow node] at (2.25,8) {};
		\node(db2)[hollow node] at (2.75,8) {};
		\node(dc2)[hollow node] at (3.25,8) {};
		\node(dd2)[hollow node] at (3.75,8) {};
		
		\draw[-,thick,dotted] (0) to (1) to (2);
		
		\draw[-,thick,dotted] (2) to (a0);
		\draw[-,thick,dotted] (2) to (b0);
		\draw[-,thick,dotted] (2) to (c0);
		\draw[-,thick,dotted] (2) to (d0);
		
		\draw[-,thick,dotted] (a0) to (a1) to (a2);
		\draw[-,thick,dotted] (b0) to (b1) to (b2);
		\draw[-,thick,dotted] (c0) to (c1) to (c2);
		\draw[-,thick,dotted] (d0) to (d1) to (d2);
		
		\draw[-,thick,dotted] (a2) to (aa0);
		\draw[-,thick,dotted] (a2) to (ab0);
		\draw[-,thick,dotted] (a2) to (ac0);
		\draw[-,thick,dotted] (a2) to (ad0);
		
		\draw[-,thick,dotted] (b2) to (ba0);
		\draw[-,thick,dotted] (b2) to (bb0);
		\draw[-,thick,dotted] (b2) to (bc0);
		\draw[-,thick,dotted] (b2) to (bd0);
		
		\draw[-,thick,dotted] (c2) to (ca0);
		\draw[-,thick,dotted] (c2) to (cb0);
		\draw[-,thick,dotted] (c2) to (cc0);
		\draw[-,thick,dotted] (c2) to (cd0);
		
		\draw[-,thick,dotted] (d2) to (da0);
		\draw[-,thick,dotted] (d2) to (db0);
		\draw[-,thick,dotted] (d2) to (dc0);
		\draw[-,thick,dotted] (d2) to (dd0);
		
		\draw[-,thick,dotted] (aa0) to (aa1) to (aa2);
		\draw[-,thick,dotted] (ab0) to (ab1) to (ab2);
		\draw[-,thick,dotted] (ac0) to (ac1) to (ac2);
		\draw[-,thick,dotted] (ad0) to (ad1) to (ad2);
		
		\draw[-,thick,dotted] (ba0) to (ba1) to (ba2);
		\draw[-,thick,dotted] (bb0) to (bb1) to (bb2);
		\draw[-,thick,dotted] (bc0) to (bc1) to (bc2);
		\draw[-,thick,dotted] (bd0) to (bd1) to (bd2);
		
		\draw[-,thick,dotted] (ca0) to (ca1) to (ca2);
		\draw[-,thick,dotted] (cb0) to (cb1) to (cb2);
		\draw[-,thick,dotted] (cc0) to (cc1) to (cc2);
		\draw[-,thick,dotted] (cd0) to (cd1) to (cd2);
		
		\draw[-,thick,dotted] (da0) to (da1) to (da2);
		\draw[-,thick,dotted] (db0) to (db1) to (db2);
		\draw[-,thick,dotted] (dc0) to (dc1) to (dc2);
		\draw[-,thick,dotted] (dd0) to (dd1) to (dd2);
		
		\draw[-,thick] (1) to (2) to (a0) to (a1) to (a2) to (ab0) to (ab1);
		\draw[-,thick] (1) to (2) to (a0) to (a1) to (a2) to (ac0) to (ac1);
		\draw[-,thick] (1) to (2) to (c0) to (c1) to (c2) to (ca0) to (ca1);
		\draw[-,thick] (1) to (2) to (c0) to (c1) to (c2) to (cd0) to (cd1);
		
		\end{tikzpicture}
  	\caption{A $C$-rake with $|C|=3$ (represented by hollow nodes, connected by interrupted lines), and a set $S \trianglelefteq R$ of height $3$ (represented by solid nodes, connected by uninterrupted lines).}\label{fig:SinR}
\end{figure}

The definition ensures that $S$ is symmetric, that strings with the same rank in $S$ are in the same block of $R$, and that all elements of $S$ have the same rank in $R$ modulo $|C|$ and therefore have the same color under $f$. See \Cref{fig:SinR} for an illustration. Observe that if $R$ has finite height, and $S \trianglelefteq R$ for some $S$ of height $\height(R)/|C|$, then for every leaf $\sigma$ of $S$ there is a unique leaf $\lambda$ of $R$ such that $\sigma \preceq \lambda$.

We proceed with two technical lemmas that will round out the set of tools we need.

\begin{lemma}\label{exists_monochromatic_set_compatible}
	Fix a computable coloring $f$ of $2^{<\omega}$, a non-empty finite set $C \subseteq \omega$, and a good $C$-rake $R$ for $f$ of finite height. For each leaf $\lambda$ of $R$, fix $c_\lambda \in C \setminus W_{f,C,\lambda}$ (where $W_{f,C,\lambda}$ is as given by \Cref{exists_dense_set_of_colors}).
	%There exists $c \in \set{c_\lambda: \lambda \text{ a leaf of } R}$ such that the following hold:
	\begin{enumerate}
		\item There exists $c \in \set{c_\lambda: \lambda \text{ a leaf of } R}$ and an $H \cong 2^{<\omega}$ monochromatic for $f$ with color $c$ such that $H^{\rank < \height(R)/|C|} \trianglelefteq R$ and $c_\lambda = c$ for every leaf $\lambda$ of $R$ extending a leaf of $H^{\rank < \height(R)/|C|}$.

		\item There is a computable procedure, uniform in canonical indices for $C$, $R$, and $\set{c_\lambda: \lambda \text{ a leaf of } R}$, that outputs some $c$ as above and a canonical index for the initial segment $H^{\rank < \height(R)/|C|}$ for some $H$ as above.
	\end{enumerate}
\end{lemma}

\begin{proof}
	%The result is clear if $R$ has height $\omega$. So assume $R$ has finite height. For each leaf $\lambda$ of $R$, fix $H_\lambda \cong 2^{<\omega}$ monochromatic for $f$ with root $\rho_\lambda \succeq \lambda$. Since $R$ is good, the color of each $H_\lambda$ is an element of $C$.
	%
	Say $|C| = k$, so that $\height(R) = mk$ for some $m \geq 1$. We define a map $g : R \to C$ by reverse induction on rank. For each element of $R$ of maximum rank, which is to say, for each leaf $\lambda$, we let $g(\lambda) = c_\lambda$. Next, fix $\sigma \in R$ that is not a leaf and suppose we have defined $g$ on all elements of $R$ of larger rank. If $\sigma$ has a unique successor $\tau$ in $R$, let $g(\sigma) = g(\tau)$. Otherwise, since $R$ is a $C$-rake, $\sigma$ has $k+1$ many successors $\tau_0,\ldots,\tau_{k}$ in $R$. We can thus computably choose a $c \in C$ such that $g(\tau_i) = g(\tau_j) = c$ for some distinct $i,j \leq k$. Let $g(\sigma) = c$. This completes the definition. Notice that for all $\sigma \preceq \tau$ in $R$, if $\block_R(\sigma) = \block_R(\tau)$ then $g(\sigma) = g(\tau)$.
	
	Now fix the color $c \in C$ that $g$ assigns to the root of $R$. We define $h : 2^{< m} \to R$ by induction on finite strings. Let $h(\emptystring)$ be the shortest $\sigma \in R$ with $f(\sigma) = c$. Since $R$ is a $C$-rake, we have that $\block_R(h(\emptystring)) = 0$, and so also $g(h(\emptystring)) = c$, by construction of~$g$. Suppose next that we have defined $h(\alpha) \in R$ for some string $\alpha$ with $|\alpha| < m-1$. Assume inductively that $\block_R(h(\alpha)) = |\alpha|$ and that $f(h(\alpha)) = g(h(\alpha)) = c$. Thus $h(\alpha)$ is not a leaf of $R$, nor in the same block as a leaf, since $\block_R(\lambda) = m-1$ for all leaves $\lambda$ of $R$. By construction of $g$, $h(\alpha)$ has two extensions $\tau,\tau' \in R$ such that $\block_R(\tau) = \block_R(\tau') = \block_R(h(\alpha)) + 1$ and $g(\tau) = g(\tau') = c$, and we may assume that $\tau,\tau'$ are of minimal length with these properties (i.e., they are the shortest elements of the next block after that of $h(\alpha)$). Let $h(\alpha0)$ and $h(\alpha 1)$ be the shortest incomparable $\sigma \succeq \tau$ and $\sigma' \succeq \tau'$ in $R$ with $f(\sigma) = f(\sigma') = c$. As above, we have for each $i < 2$ that $\block_R(h(\alpha i)) =  \block_R(h(\alpha)) + 1 =|\alpha| + 1$ and $f(h(\alpha i)) = g(h(\alpha i)) = c$, so the inductive assumptions are maintained.
	
	Let $S = \ran(h)$. Then $S \cong 2^{< m}$ as witnessed by $h$, and $S$ is monochromatic for $f$ with color $c$. It follows from the construction that $S \trianglelefteq R$. Consider any leaf $\sigma$ of $S$, and let $\lambda$ be the unique leaf of $R$ extending $\sigma$. Since $\block_R(\sigma) = \block_R(\lambda)$, we have that $c = f(\sigma) = g(\sigma) = g(\lambda) = c_{\lambda}$. Moreover, since $c = c_{\lambda} \in C \setminus W_{C,f,\lambda}$, it follows by \Cref{exists_dense_set_of_colors}~(2) that $\set{\rho \in 2^{<\omega} : f(\rho) = c}$ is dense above some extension of $\lambda$.
	
	For each leaf $\lambda$ of $R$ extending a leaf of $S$ there therefore exists an $H_\lambda \cong 2^{<\omega}$ monochromatic for $f$ with color $c$ and with root  $\rho_\lambda \succeq \lambda$. Define
	\[
		H = S \cup \left( \bigcup_{\substack{\lambda \text{ a leaf of } R\\ \text{ extending a}\\\text{leaf of } S}} H_\lambda \setminus \set{\rho_{\lambda}} \right).
	\]
	Then $H \cong 2^{<\omega}$, $H$ is monochromatic for $f$ with color $c$, and $H^{\rank < \height(R)/|C|} = H^{\rank < m} = S$ satisfies the conclusion of part (1).
	
	Finally, note that $k$, $m$, and the function $g$ are all uniformly computable from canonical indices for $C$, $R$, and the set $\set{c_\lambda: \lambda \text{ a leaf of } R}$. Hence, so is the color assigned by $g$ to the root of $R$, which is $c$, and from here we can uniformly compute the function $h : 2^{< m} \to R$. Since $H^{\rank < \height(R)/|C|} = \ran(h)$, and $h$ is strictly increasing under $\preceq$, this completes the proof.
\end{proof}

\begin{lemma}\label{exists_infinite_rake2}
	Fix a computable coloring $f$ of $2^{<\omega}$ and a $\Sigma^0_1$ property $\varphi$ of finite subsets of $2^{<\omega}$ such that whenever $H \cong 2^{<\omega}$ is monochromatic for $f$ then $\varphi(H^{\rank <n})$ holds for some $n$. There exists a non-empty finite set $C \subseteq \omega$ and a good $C$-rake $R$ for $f$ of finite height such that for every $H \cong 2^{<\omega}$ monochromatic for $f$ with $H^{\rank < \height(R)/|C|} \trianglelefteq R$ we have that $\varphi(H^{\rank < \height(R)/|C|})$ holds.
\end{lemma}

\begin{proof}
	Apply \Cref{exists_infinite_rake} to obtain a non-empty finite set $C \subseteq \omega$ and a computable good $C$-rake $R_0$ for $f$ of height $\omega$. Let $\mathcal{P}$ be the class of all $H \trianglelefteq R_0$ such that $H$ has height $\omega$ and is monochromatic for $f$. Notice that $\mathcal{P}$ is a $\Pi^0_1$ class. The only part of the definition for which this may not be immediately apparent is the stipulation that each $H \in \mathcal{P}$ has height $\omega$, and so in particular, that each element of $H$ has a pair of incomparable successors in $H$. But since we want $H \trianglelefteq R_0$, the latter condition can be expressed as
	\[
		(\forall \sigma \in H)(\exists \tau,\tau' \in H)[\sigma \preceq \tau,\tau' \wedge \tau \mid \tau' \wedge \block_{R_0}(\tau) = \block_{R_0}(\tau') = \block_{R_0}(\sigma)+ 1].
	\]
	And since the set of all elements in $R_0$ in any given block is finite, the above is $\Pi^0_1$.
	
	Since $R_0$ has height $\omega$, it follows that for each $c \in C$ there exists $H \in \mathcal{P}$  monochromatic for $f$ with color $c$, so in particular $\mathcal{P} \neq \emptyset$. Now, let $\mathcal{Q}$ be the class of all $H \in \mathcal{P}$ such that $\neg \varphi(H^{\rank < n})$ holds for all~$n$. This is again a $\Pi^0_1$ class, and by hypothesis, $\mathcal{Q} = \emptyset$. By compactness of Cantor space, we can fix an $m \in \omega$ such that for every $H \in \mathcal{P}$ there exists $n \leq m$ for which $\varphi(H^{\rank < n})$ holds. Let $R = R_0^{\rank < m|C|}$, a good $C$-rake of finite height. If $H \cong 2^{<\omega}$ is any set monochromatic for $f$ and satisfying $H^{\rank < m} \trianglelefteq R$, then there must be a $G \in \mathcal{P}$ so that $H^{\rank < m} = G^{\rank < m}$. Thus, $\varphi(H^{\rank < n})$ holds for some $n \leq m = \height(R)/|C|$, as was to be shown. By our use conventions (\Cref{sec:background}), this implies that $\varphi(H^{\rank < \height(R)/|C|})$ holds as well.
\end{proof}

We can now assemble the pieces to prove our main theorem.

\begin{theorem}\label{TT1N_theorem_main}
	$\fopart \TT^1_{\mathbb{N}} \ured \RT^1_{\mathbb{N}}$.
\end{theorem}

\begin{proof}
	Consider an arbitrary instance $\seq{A,\Delta,\Gamma}$ of $\fopart\TT^1_{\mathbb{N}}$. We have that $\Delta(A)$ is a coloring $f$ of $2^{<\omega}$, and if $H \cong 2^{<\omega}$ is any monochromatic set for $f$ then $\Gamma(A, H)(0) \converges$, and the output of this computation is a solution to $\seq{A,\Delta,\Gamma}$. We define an instance $d$ of $\RT^1_{\mathbb{N}}$, uniformly computable from $\seq{A,\Delta,\Gamma}$, such that if $M$ is any infinite homogeneous set for $d$ then $\seq{A,\Delta,\Gamma} \oplus M$ uniformly computes $\Gamma(A, H)(0)$ for some such $H$. 
	
	To begin, we claim that there exists a non-empty finite set $C \subseteq \omega$, a good $C$-rake $R$ for $f$ of finite height, and an $s \in \omega$ with the following properties. For each leaf $\lambda$ of $R$, and each choice of $c_\lambda \in C \setminus W_{f,C,\lambda}[s]$, we apply \Cref{exists_monochromatic_set_compatible}~(2), relativized to $A$. This specifies an $A$-computable procedure for determining a $c \in \set{c_\lambda: \lambda \text{ a leaf of } R}$ and a finite set $S \trianglelefteq R$ of height $\height(R)/|C|$ that is monochromatic for $f$ with color $c$ and such that $c_\lambda = c$ for every leaf $\lambda$ of $R$ extending a leaf of $S$. The properties of $C$, $R$, and $s$ are that this procedure halts, and that for the resulting $S$ we have that $\Gamma(A,S)(0) \converges$.
	
	Note that the procedure from the lemma need not halt for arbitrary $s$, since we may have $W_{f,C,\lambda}[s] \subset W_{f,C,\lambda}$ for some $\lambda$ and therefore it could be that some $c_\lambda$ actually belongs to $W_{f,C,\lambda}$. (Thus, the hypotheses of the lemma would not be met.) Likewise, nothing about the lemma guarantees that if the procedure halts and outputs $c$ and $S$ then necessarily $\Gamma(A,S)(0) \converges$.
	
	To prove the claim, note that if we take $\varphi(X)$ to be the $\Sigma^0_1(A)$ formula stating that $\Gamma(A,X)(0) \converges$, then $f$ and $\varphi$ are precisely as in the hypothesis of \Cref{exists_infinite_rake2}, relativized to~$A$. Thus, there exists a non-empty finite set $C \subseteq \omega$ and a good $C$-rake $R$ for $f$ of finite height such that for for every $H \cong 2^{<\omega}$ monochromatic for $f$ with $H^{\rank < \height(R)/|C|} \trianglelefteq R$ we have $\Gamma(A,H^{\rank < \height(R)/|C|}) \converges$.
	
	At the same time, by \Cref{exists_monochromatic_set_compatible}, for every choice of $c_\lambda \in C \setminus W_{f,C,\lambda}$ for $\lambda$ a leaf of $R$, there is a $c \in \set{c_\lambda: \lambda \text{ a leaf of } R}$ and an $H \cong 2^{<\omega}$ monochromatic for $f$ with color $c$ such that $H^{\rank < \height(R)/|C|} \trianglelefteq R$ and $c_\lambda = c$ for every leaf $\lambda$ of $R$ extending a leaf of $H \cong 2^{<\omega}$. Thus, in particular for this $H$, we have that $\Gamma(A,H^{\rank < \height(R)/|C|}) \converges$. But $c$ and $H^{\rank < \height(R)/|C|}$ are the number and finite set produced by the computable procedure in part (2) of that lemma. Hence, if $s$ is large enough so that $W_{f,C,\lambda}[s] = W_{f,C,\lambda}$ for all leaves $\lambda$ of $R$, then $C$, $R$, and $s$ are as desired. The claim is proved.
	
	The definition of $C$, $R$, and $s$ in the statement of the claim is uniformly $\Delta^0_2$ in~$\seq{A,\Delta,\Gamma}$. Hence, uniformly computably in $\seq{A,\Delta,\Gamma}$, we can approximate the least such $C$, $R$, and $s$ by some $\seq{C_t: t \in \omega}$, $\seq{R_t: t \in \omega}$, and $\seq{s_t: t \in \omega}$. That is, for each $t$, $C_t$ and $R_t$ are canonical indices of finite sets, $s_t$ is a number, and we have that $\lim_t C_t = C$, $\lim_t R_t = R$, and $\lim_t s_t = s$.  Of course, $R_t$ need not be good, or even a $C_t$-rake, other than in the limit. Let $n_t$ denote the number of leaves of $R_t$.
	
	By \Cref{exists_dense_set_of_colors}~(1), relativized to $A$, we have that for each $t \in \omega$ and $i < n_t$ the set $W_{f,C_t,\lambda_i}$ is $\Sigma^0_2(A)$ uniformly in $f$, $C_t$ and $\lambda_i$. This means $W_{f,C_t,\lambda_i}$ can be $A$-computably approximated by a sequence of finite sets $\seq{V_{t,i,u} : u \in \omega}$, uniformly in $C_t$ and $\lambda_i$. More precisely, $V_{t,i,u} \subseteq C_t$ for all $u$, and for each $c \in C_t$ we have that $c \in V_{t,i,u}$ for cofinitely many $u$ if and only if $c \in W_{f,C_t,\lambda_i}$. By removing the least element whose membership in $V_{t,i,u}$ has changed most recently (with respect to $u$), or the least element if none such exists, we may further assume that $C_t \setminus V_{t,i,u} \neq \emptyset$ for all $u$. (If $t$ happens to be large enough that $C_t = C$ and $R_t = R$, then $C_t \setminus W_{f,C_t,\lambda_i} \neq \emptyset$ since $R$ is a good $C$-rake for $f$. Then since $V_{t,i,u} \supseteq W_{f,C_t,\lambda_i}$ for all sufficiently large $u$, any such element removed will indeed belong to $C_t \setminus V_{t,i,u}$ for any such $u$.)
	
	We now construct a coloring $d : \omega \to \omega$. (We will argue below that this is actually an instance of $\RT^1_{\mathbb{N}}$.) Fix $t$; we define $d(t)$. First, compute $C_t$, $R_t$, and $s_t$, and let $\lambda_0,\ldots,\lambda_{n_t-1}$ be the leaves of $R_t$, enumerated in lexicographic order.
	%By \Cref{exists_dense_set_of_colors}~(1), relativized to $A$, we have that for each $i < n_t$ the set $W_{f,C_t,\lambda_i}$ is $\Sigma^0_2(A)$ uniformly in $f$, $C_t$ and $\lambda_t$. This means $W_{f,C_t,\lambda_i}$ can be $A$-computably approximated by a sequence of finite sets $\seq{V_u : u \in \omega}$, uniformly in $C_t$ and $\lambda_t$. More precisely, $V_u \subseteq C_t$ for all $u$, and for each $c \in C$ we have that $c \in V_u$ for cofinitely many $u$ if and only if $u \in W_{f,C_t,\lambda_i}$. By removing the least element whose membership in $V_u$ has changed most recently (with respect to $u$), or the least element if none such exists, we may further assume that $C_t \setminus V_u \neq \emptyset$ for all $u$. (If $t$ happens to be large enough that $C_t = C$ and $R_t = R$, then $C_t \setminus W_{f,C_t,\lambda_i} \neq \emptyset$ since $R$ is a good $C$-rake for $f$. So any such element removed will indeed belong to $C_t \setminus V_u$ for all large enough $u$.)
	Then, for each $i < n_t$, let $c_i$ be the least element of $C_t \setminus V_{t,i,t}$, which is non-empty as discussed above. Finally, we define $d(t) = \seq{C_t,R_t,c_0,\ldots,c_{n_t-1}}$. Clearly, $d$ is uniformly computable from $\seq{A,\Delta,\Gamma}$.
	
	Now fix $t_0$ so that $C_t = C$, $R_t = R$, and $s_t = s$ for all $t \geq t_0$. Let $n = n_{t_0}$, the number of leaves of $R$, and enumerate the leaves in lexicographic order as $\lambda_0,\ldots,\lambda_{n-1}$. For each $t \geq t_0$ and $i < n$, the approximations $\seq{V_{t,i,u} : u \in \omega}$ to $W_{f,C_t,\lambda_i} = W_{f,C,\lambda_i}$ now only depend on $i$ (and not $t$). Thus, we may fix $t_1 \geq t_0$ so that additionally, for every $i < n$ and every $t \geq t_1$, we have $V_{t,i,t} \supseteq W_{f,C,\lambda_i}$. By construction of $d$, we have that if $t \geq t_1$ then $d(t) = \seq{C,R,c_0,\ldots,c_{n-1}}$ for some $c_0,\ldots,c_{n-1}$ with $c_i \in C \setminus V_{t,i,t}$ for each $i < n$. The value of each $c_i$ may depend on $t$, but since $V_{t,i,t} \supseteq W_{f,C,\lambda_i}$ we have that $c_i \in C \setminus W_{f,C,\lambda_i}$. Hence, the range of $d$ is bounded, so $d$ is an instance of $\RT^1_{\mathbb{N}}$, as desired.
	
	Now consider any infinite homogeneous set $M \subseteq \omega$ for $d$. As noted above, the color of $M$ under $d$ must be $\seq{C,R,c_0,\ldots,c_{n-1}}$, where $c_i \in C \setminus W_{f,C,\lambda_i}$ for each $i < n$. By the choice of $C$ and $R$, we can apply the $A$-computable procedure of \Cref{exists_monochromatic_set_compatible} to find a $c \in \set{c_i : i < n}$ and a finite set $S \trianglelefteq R$ of height $\height(R)/|C|$ that is monochromatic for $f$ with color $c$ and such that $c_i = c$ for every $\lambda_i$ that extends a leaf of $S$, and $\Gamma(A,S)(0) \converges$.
	
	We now show that $\Gamma(A,S)(0)$ is a $\fopart \TT^1_{\mathbb{N}}$-solution to $\seq{A,\Delta,\Gamma}$. Since $c_i \in C \setminus W_{f,C,\lambda_i}$ for each $i < n$, it follows by \Cref{exists_dense_set_of_colors}~(2) that the set $\set{\rho \in 2^{<\omega} : f(\rho) = c_i}$ is dense above some extension of $\lambda_i$. In particular, this is true for the $\lambda_i$ extending the leaves of $S$, for which $c_i = c$. We can now argue as in the ending of the proof of \Cref{exists_monochromatic_set_compatible} to conclude that there is an $H \cong 2^{<\omega}$ monochromatic for $f$ with color $c$ and such that $H^{\rank < \height(R)/|C|} = S$. Hence, $\Gamma(A,H)(0)$ is a $\fopart \TT^1_{\mathbb{N}}$-solution to $\seq{A,\Delta,\Gamma}$, and by our use conventions, $\Gamma(A,H)(0) = \Gamma(A,H^{\rank < \height(R)/|C|})(0) = \Gamma(A,S)(0)$, as desired. This completes the proof.
\end{proof}

\begin{corollary}\label{cor:TT1N_equiv_RT1N}
	$\fopart \TT^1_{\mathbb{N}} \uequiv \RT^1_{\mathbb{N}}$.
\end{corollary}

Since, trivially $\fopart \TT^1_k \ured \TT^1_{\mathbb{N}}$ for all $k \geq 2$, but $\TT^1_k \nured \RT^1_{\mathbb{N}}$ by \Cref{thm:TT1k_not_to_RT1N}, we immediately obtain the following.

\begin{corollary}\label{TT12_not_FO}
	For every $j,k \geq 2$, $\TT^1_k \nured \fopart \TT^1_j$. In particular, $\TT^1_k$ is not Weihrauch equivalent to any first-order problem.
\end{corollary}

\noindent Thus we also have the analogue, in the Weihrauch degrees, of the remark made at the end of \Cref{sec:intro} that $\TT^1$ is not equivalent to any $\Pi^1_1$ statement over $\RCA_0$.

\section{Questions}

We conclude with a few questions left over from our analysis, some concrete and some more open-ended.
%As mentioned in \Cref{sec:boundedcols}, we do not know the answer to the following.
%\begin{question}\label{Q1}
%	Is it the case that $\fopart \TT^1_2 \ured \RT^1_+$?	
%\end{question}
%Another
One line of inquiry concerns clarifying the relationship between the problems studied here and various weaker forms of the infinite pigeonhole principle. A notable example, ubiquitous in the literature, is closed choice on $k$-element sets, $\mathsf{C}_k$ (cf.~\cite[Section 7]{BGP-2021}.) Recently, Pauly \cite{Pauly-TA} has shown that $\mathsf{C}_3 \nured \TT^1_2$. By results of Brattka and Rakotoniaina~\cite[Proposition 7.4]{BR-2017} is known than $\mathsf{C}_3 \Ured \RT^1_3$ and that $\mathsf{C}_3 \nured \RT^1_2$. Pauly's theorem therefore improves our \Cref{RT1more_non_TT1less} for $k = 3$ and $j = 2$. Some other related results appear in recent work of Gill~\cite{Gill-TA}, but many basic questions remain open.

In \Cref{cor:TT1N_equiv_RT1N} we obtained an answer to the analogue of the longstanding question, discussed in the introduction, of whether the first-order part of $\TT^1$ is $\RT^1$. This relied on \Cref{TT1N_theorem_main}, and on the machinery of rakes used in its proof. This raises the following question.

\begin{question}\label{Q3}
	Can the combinatorics deployed in the proof of \Cref{TT1N_theorem_main} be adapted and applied to study the strength of $\TT^1$ in second-order arithmetic?
\end{question}

\noindent Recently, Chong, Li, Wang, and Yang~\cite[Theorem 4.1]{CLWY-2020} showed that $\TT^1$ is conservative over $\RCA_0 + \RT^1 + \mathsf{P}\Sigma^0_1$. Here, $\mathsf{P}\Sigma^0_1$ (cf.~\cite[Definition 4.4]{CLWY-2020}) is a certain arithmetical principle equivalent, by results of Kreuzer and Yokoyama~\cite[Theorem 1.3]{KY-2016}, to the totality of the Ackermann--P\'{e}ter function relativized to any strictly monotonic function. Whether the first-order part of $\TT^1$ is $\RT^1$ is thus equivalent to whether $\mathsf{P}\Sigma^0_1$ can be eliminated in this result. In connection with \Cref{Q3}, we can ask the following.

\begin{question}
	Does $\mathsf{P}\Sigma^0_1$ admit a natural form (or forms) as an instance-solution problem? How does it (or, do they) compare under Weihrauch reducibility with the various versions of $\TT^1$ studied in this article?
\end{question}

Lastly, we ask a general question, implicit in earlier papers like~\cite{DSY-TA} and~\cite{SV-2022}.

\begin{question}\label{Q4}
	What is the precise relationship between the first-order parts of problems in the context of Weihrauch reducibility, and the first-order parts of principles in the context of reverse mathematics?
\end{question}

\noindent Our results here add to the body of evidence showing that some sort of connection between the two notions exists. Whether \Cref{Q4} can be answered for a certain class of problems/principles, or whether the notion of first-order part under Weihrauch reducibility needs to be further refined to make some kind of answer possible, is unknown. In our view, clarifying this situation would yield important foundational insight, and further investigation is warranted.

\end{document}